\documentclass[12pt]{article}
\usepackage{amssymb,amsmath,amsthm, amsfonts}
\usepackage{graphicx}
\usepackage{epsfig}
\usepackage{tikz}
\usepackage{mathrsfs}

\def\disp{\displaystyle}

\def\dref#1{(\ref{#1})}

\theoremstyle{plain}
\newtheorem{theorem}{Theorem}[section]
\newtheorem{lemma}{Lemma}[section]

\theoremstyle{definition}
\newtheorem{definition}{Definition}[section]

\newtheorem{remark}{Remark}[section]

\numberwithin{equation}{section}

\begin{document}

\title{\bf Global existence to a  $3D$ chemotaxis-Navier-stokes system with  nonlinear diffusion and rotation}

\author{
Jiashan Zheng$^{1}$\thanks{Corresponding author.   E-mail address:
 zhengjiashan2008@163.com (J.Zheng)}, Yanyan Li$^{1}$,Xinhua Zou$^{2}$, Dongfang Zhang$^{1}$,   Weifang Yan$^{1}$
 \\
    $^{1}$School of Mathematics and Statistics Science,\\
     Ludong University, Yantai 264025,  P.R.China \\
 $^2$School of Food Engineering,
Ludong University, Yantai 264025,  P.R.China \\
}
\date{}


\maketitle \vspace{0.3cm}
\noindent
\begin{abstract}
This paper is concerned with the following quasilinear chemotaxis--Navier--Stokes
 system
with  nonlinear diffusion and rotation
$$
 \left\{
 \begin{array}{l}
   n_t+u\cdot\nabla n=\Delta n^m-\nabla\cdot(nS(x,n,c)\cdot\nabla c),\quad
x\in \Omega, t>0,\\
    c_t+u\cdot\nabla c=\Delta c-nc,\quad
x\in \Omega, t>0,\\
u_t+\kappa(u \cdot \nabla)u+\nabla P=\Delta u+n\nabla \phi ,\quad
x\in \Omega, t>0,\\
\nabla\cdot u=0,\quad
x\in \Omega, t>0\\
 \end{array}\right.\eqno(CNF)
 $$
 is considered under the no-flux boundary conditions for $n, c$ and the Dirichlet boundary condition for
$u$  in a three-dimensional convex  domain $\Omega\subseteq \mathbb{R}^3$ with smooth boundary, which describes the motion of oxygen-driven bacteria in a fluid. Here
  $\kappa\in \mathbb{R}$ and  $S$ denotes the strength of nonlinear fluid convection and a given tensor-valued function, respectively.
  Assume $m>\frac{10}{9}$ and  $S$ fulfills
$|S(x,n,c)| \leq  S_0(c)$ for all $(x,n,c)\in \bar{\Omega} \times [0, \infty)\times[0, \infty)$ with $S_0(c)$
nondecreasing on $[0,\infty)$, then for any
 reasonably regular initial data,
 the corresponding initial-boundary problem $(CNF)$ admits at least one global weak solution.

%

%

\end{abstract}

\vspace{0.3cm}


\newpage
\section{Introduction}

In biological contexts, many simple life-forms exhibit a complex collective behavior.
This effect, called chemotaxis, is presumed to have a deep impact on the time evolution of a bacteria population.
The chemotaxis system is proposed by Keller and Segel \cite{Keller2710} in 1970. During the past four decades, the
chemotaxis system,  the chemical substrate
can be produced or consumed by the cells
 has been well studied in mathematical biology (\cite{Bellomo1216,Horstmann791,Tao2,Winkler21215}).
 For the more related works in this direction, we mention that a corresponding quasilinear version has been deeply investigated by \cite{Tao794,Wang79k1,Zheng00,Zheng33312186,Zhengddffsdsd6}.

Furthermore, in \cite{Dombrowski3344} it can be observed experimentally that bacteria are suspended in the fluid, which is influenced by the gravitational forcing generated by the aggregation of cells. Then the movement of bacteria is effected by composite factors, namely, random diffusion, chemotactic migration towards gradients of oxygen and transport through the fluid. Taking into account all these processes, Tuval et al. (\cite{Tuval1215}) proposed the model
\begin{equation}
 \left\{\begin{array}{ll}
   n_t+u\cdot\nabla n=\Delta n -C_S\nabla\cdot( nS(c)\cdot\nabla c),\quad
x\in \Omega, t>0,\\
    c_t+u\cdot\nabla c=\Delta c-nf(c),\quad
x\in \Omega, t>0,\\
u_t+\kappa(u\cdot\nabla u)=\nabla P+\Delta u+n\nabla \phi ,\quad
x\in \Omega, t>0,\\
\nabla\cdot u=0,\quad
x\in \Omega, t>0\\
 \end{array}\right.\label{ddfddfghhsdertyyuswddferrttbnmkl1.1}
\end{equation}
for the unknown bacterial density $n$, the oxygen concentration $c$, the fluid velocity field $u$ and the
associated pressure $P$ in the physical domain  $\Omega\subset \mathbb{R}^N.$ Here, $\kappa\in R$ is related to the strength of nonlinear fluid convection, the functions $S(c)$, $f(c)$ and $\phi$ denotes  the chemotactic
sensitivity,   the consumption rate of the oxygen by the bacteria and the gravitational
potential, respectively.
System \dref{ddfddfghhsdertyyuswddferrttbnmkl1.1} describes the movement of the cells towards the higher concentration of the oxygen that is consumed by the cells.
For system \dref{ddfddfghhsdertyyuswddferrttbnmkl1.1}, by making use of {\bf energy-type functionals},
there have been many literatures studied  the existence of global solutions in the bounded
domain or the whole space under some assumption on $f (c)$, $S(c)$ and initial data (see \cite{Bellomo1216,Chaexdd12176,Duan12186,Liu1215,Winklerssddcf31215,Winklercvb12176,Winkler51215,Zhangcvb12176} and references therein).
In fact, in the {\bf two-dimensional setting}, if  $\kappa=0,$  Duan et al. (\cite{Duan12186}) proved global existence of weak solutions for the Cauchy problem of \dref{ddfddfghhsdertyyuswddferrttbnmkl1.1}, under smallness assumptions on either $\nabla\phi$
 or the initial data for oxygen concentration. 
 Winkler \cite{Winklercvb12176}
 proved that \dref{ddfddfghhsdertyyuswddferrttbnmkl1.1}
 has a unique global classical solution in a bounded convex domain $\Omega\subset \mathbb{R}^2$
 with smooth boundary for large data with suitable regularity.
While for {\bf three-dimensional} chemotaxis(-Navier)-Stokes system \dref{ddfddfghhsdertyyuswddferrttbnmkl1.1},
  when $\kappa=0,$ Winkler \cite{Winklercvb12176} proved that the chemotaxis-Stokes system \dref{ddfddfghhsdertyyuswddferrttbnmkl1.1} possesses
  at least one global weak solution. 

When $\Delta n$ is
replaced by $\nabla\cdot(D(n)\nabla n)$ in the first equation in \dref{ddfddfghhsdertyyuswddferrttbnmkl1.1},  some authors used some {\bf  energy-type functionals} to prove the global or local existence of the solutions to system \dref{ddfddfghhsdertyyuswddferrttbnmkl1.1}
(see \cite{Duanx41215,Francesco791,Zhangddffcvb12176} and references therein).  Indeed,
 the porous medium
diffusion function $D$ satisfies
\begin{equation}\label{ghnjmk9161gyyhuug}
D\in C^\iota_{loc} ([0,\infty))~~\mbox{for some}~~\iota>0,~~~
C_{\bar{D}}n^{m-1}\geq D(n)\geq C_Dn^{m-1}~~ \mbox{for all}~~ n>0
\end{equation}
with some $m >0$ and $C_{\bar{D}}\geq C_D$,
Zhang and Li (\cite{Zhangddffcvb12176})
 used some energy-type functionals to prove the global   existence of the weak solutions to system \dref{ddfddfghhsdertyyuswddferrttbnmkl1.1} when $m\geq\frac{2}{3}$.
Recently, if $\kappa= 0$, Tao and Winkler (\cite{Tao71215})  proved the locally bounded global of weak solution of \dref{ddfddfghhsdertyyuswddferrttbnmkl1.1} in $\mathbb{R}^3$ as $m > \frac{8}{7}$.
 These energy-type functionals play  key roles in their proofs.

Generally speaking, more recent modeling approaches (see DiLuzio  et al. \cite{DiLuzio1215}, Winkler \cite{Winklerdfvg61215,Winkler11215}, Xue et al.
\cite{Xue1215}) suggest that chemotactic migration is not directed to the gradient of the chemical substance but with a rotation, and that accordingly, the chemotactic sensitivity should be a tensor which may have nontrivial off-diagonal entries.
Motivated by the above works, we will investigate the chemotaxis-Navier-stokes system with  nonlinear diffusion and  the rotational sensitivity in this paper. Precisely, we shall consider the following initial-boundary problem
\begin{equation}
 \left\{\begin{array}{ll}
   n_t+u\cdot\nabla n=\nabla\cdot(D(n)\nabla n) -\nabla\cdot( nS(x,n,c)\cdot\nabla c),\quad
x\in \Omega, t>0,\\
    c_t+u\cdot\nabla c=\Delta c-nc,\quad
x\in \Omega, t>0,\\
u_t+\kappa(u \cdot \nabla)u+\nabla P=\Delta u+n\nabla \phi ,\quad
x\in \Omega, t>0,\\
\nabla\cdot u=0,\quad
x\in \Omega, t>0,\\
 \disp{( D(n)\nabla n-nS(x,n,c)\cdot\nabla c)\cdot\nu=\nabla c\cdot\nu=0,u=0,}\quad
x\in \partial\Omega, t>0,\\
\disp{n(x,0)=n_0(x),c(x,0)=c_0(x),u(x,0)=u_0(x),}\quad
x\in \Omega,\\
 \end{array}\right.\label{1.1}
\end{equation}
where  $\Omega \subseteq \mathbb{R}^3 $ be a bounded  convex  domain with smooth boundary, $S(x,n,c)$ is a tensor-valued function, satisfying
\begin{equation}\label{x1.73142vghf48rtgyhu}
S\in C^2(\bar{\Omega}\times[0,\infty)^2;\mathbb{R}^{3\times3})
 \end{equation}
 and
 \begin{equation}\label{x1.73142vghf48gg}|S(x, n, c)|\leq  S_0(c) ~~~~\mbox{for all}~~ (x, n, c)\in\Omega\times [0,\infty)^2
 \end{equation}
with  some nondecreasing $S_0 : [0,\infty)\rightarrow \mathbb{R},$
which indicates the rotational effect.
$D,$ $\nabla\phi,$ $\kappa$, $n(x, t)$, $u(x, t)$,$c(x, t)$ and $P(x, t)$ are denoted as before.
Due to the significance of the biological background, many mathematicians have studied \dref{1.1}
and made more progress in the past years (Ishida \cite{Ishida1215}, Zheng \cite{Zhengsdsd6,Zhengssssssssdefr23}, Wang et al. \cite{Wang11215,Wang21215}, Winkler \cite{Winkler11215}, Wang and Li \cite{Wangssderr79k1}, Cao and Lankeit \cite{Caoddf22119}).
In contrast to the chemotaxis-(Navier-)Stokes system
\dref{ddfddfghhsdertyyuswddferrttbnmkl1.1}, chemotaxis-(Navier-)Stokes system \dref{1.1} with tensor-valued sensitivity loses some natural gradient-like structure, which  gives rise to
considerable mathematical difficulties and some new analysis is needed.
Therefore, as far as I know that
only very few results appear to be available on chemotaxis-(Navier-)Stokes with such tensor-valued
sensitivities.
To this end,  if
$\kappa = 0$ in \dref{1.1} and  $D$ satisfies \dref{ghnjmk9161gyyhuug} with
$m > \frac{7}{6}$,
Winkler (\cite{Winkler11215}) showed  that the {\bf three} space dimensions of the chemotaxis--Stokes system ($\kappa = 0$ in \dref{1.1}) possessed at least one bounded weak solution
which stabilizes to the spatially homogeneous equilibrium  $(\bar{n}_0, 0, 0)$ with $\bar{n}_0:=\frac{1}{|\Omega|}\int_{\Omega}n_0$ as $t\rightarrow\infty$.
While,
if $\kappa\neq0,$
assuming that 
\dref{x1.73142vghf48rtgyhu}--\dref{x1.73142vghf48gg} hold and $D(n) = m n^{m-1}$,
 Ishida (\cite{Ishida1215}) showed  that
 \dref{1.1}
admits
a bounded global weak solution in {\bf two} space dimensions.
More recently, if $D\equiv1$, $S$ satisfies that \dref{x1.73142vghf48rtgyhu}--\dref{x1.73142vghf48gg} and  the initial data satisfied  certain {\bf smallness
conditions}, Cao and Lankeit (\cite{Caoddf22119}) proved that \dref{1.1} possessed
a global classical solution and gave the decay properties of these solutions on {\bf three space dimensions}.
However, for three space dimensions of full nonlinear chemotaxis-Navier-Stokes system  \dref{1.1} without
the assuming of {\bf smallness
conditions}, there is still a open problem.

Before formulating our main results, we first explain the notations and conventions used
throughout this paper.
Throughout this paper, let $A_{r}$ denote the Stokes operator with domain $D(A_{r}) := W^{2,{r}}(\Omega)\cap  W^{1,{r}}_0(\Omega)
\cap L^{r}_{\sigma}(\Omega)$,
and
$L^{r}_{\sigma}(\Omega) := \{\varphi\in  L^{r}(\Omega)|\nabla\cdot\varphi = 0\}$ for ${r}\in(1,\infty)$
 (\cite{Sohr}).

\begin{theorem}\label{theorem3}
Let
\begin{equation}
\phi\in W^{1,\infty}(\Omega).
\label{dd1.1fghyuisdakkkllljjjkk}
\end{equation}
Moreover, assume that the initial data
$(n_0, c_0, u_0)$ satisfy
\begin{equation}\label{ccvvx1.731426677gg}
\left\{
\begin{array}{ll}
\displaystyle{n_0\in C^\kappa(\bar{\Omega})~~\mbox{for certain}~~ \kappa > 0~~ \mbox{with}~~ n_0\geq0 ~~\mbox{in}~~\Omega},\\
\displaystyle{c_0\in W^{1,\infty}(\Omega)~~\mbox{with}~~c_0\geq0~~\mbox{in}~~\bar{\Omega},}\\
\displaystyle{u_0\in D(A^\gamma_{r})~~\mbox{for~~ some}~~\gamma\in ( \frac{3}{4}, 1)~~\mbox{and any}~~ {r}\in (1,\infty).}\\
\end{array}
\right.
\end{equation}
and suppose that $m$ and $S$ satisfy \dref{ghnjmk9161gyyhuug} and \dref{x1.73142vghf48rtgyhu}--\dref{x1.73142vghf48gg}, respectively.
 If
\begin{equation}\label{x1.73142vghf48}m>\frac{10}{9},
\end{equation}
then it holds that
there exists at least one global weak
solution (in the sense of Definition \ref{df1} above) of problem \dref{1.1}. 
\end{theorem}
\begin{remark}
(i) If $ S(x,n,c):=C_S$ and $\kappa=0,$  Theorem \ref{theorem3}   extends the results of Theorem 1.1 of Tao and Winkler \cite{Tao71215}, who proved the possibility of $\mathbf{locally~~ bounded ~~global~~ solutions}$,
in the case that $m> \frac{8}{7}$.

(ii) If  $\kappa=0,$  Theorem \ref{theorem3}   extends the results of Theorem 1.1 of Zheng \cite{Zhengssssssssdefr23}, who proved the possibility of $\mathbf{locally~~ bounded ~~global~~ solutions}$,
in the case that $m> \frac{9}{8}$.



(iii) In view of Theorem \ref{theorem3}, if the flow of fluid is ignored or the fluid is
stationary in \dref{1.1}, $ S(x,n,c):=C_S,$ and $N=3$,  Theorem \ref{theorem3}
is consistent with the result of Theorem 2.1 of  Zheng and Wang(\cite{Zhengddffsdsd6}),
who proved the possibility of $\mathbf{global~~ existence}$,
in the case that $m> \frac{9}{8}$.

(iv) If $m>\frac{10}{9}$,  Theorem \ref{theorem3} is hold without
requirement of 
 the small initial data (see Cao and Lankeit \cite{Caoddf22119}).

\end{remark}

Before proving our main results about the model  \dref{1.1} in the next part, let us mention the following Keller-Segel-(Navier)-Stokes model (accounting for terms $+n-c$ in place of $-nc$ in the second equation of \dref{1.1}), which is a closely related variant of \dref{1.1}
 \begin{equation}
 \left\{\begin{array}{ll}
   n_t+u\cdot\nabla n=\nabla\cdot(  D(n)\nabla n)-\nabla\cdot(nS(x,n,c)\nabla c),\quad
x\in \Omega, t>0,\\
    c_t+u\cdot\nabla c=\Delta c-c+n,\quad
x\in \Omega, t>0,\\
u_t+\kappa (u\cdot\nabla)u+\nabla P=\Delta u+n\nabla \phi,\quad
x\in \Omega, t>0,\\
\nabla\cdot u=0,\quad
x\in \Omega, t>0.\\
 \end{array}\right.\label{1ssxdcfddfrgtvgb.1}
\end{equation}
In contrast to \dref{1.1}, in the classical Keller-Segel system the chemoattractant is produced by the bacteria
themselves and not consumed,
and models of Keller-Segel-
(Navier)-Stokes type have also been considered (see Wang and Xiang \cite{Wang21215,Wangss21215}, Liu and Wang \cite{Liu334231215} Zheng \cite{Zheddfffggngddffsdsd6}).


The crucial step of our approaches  is to establish 
%
%
 the  natural gradient-like energy functional
%
%
\begin{equation}
 \left\{\begin{array}{ll}
 \disp\int_{\Omega}n_{\varepsilon}(\cdot,t)\ln n_{\varepsilon}(\cdot,t)+\int_{\Omega}|\nabla\sqrt{c_{\varepsilon}}(\cdot,t)|^2+\int_{\Omega}|u_{\varepsilon}(\cdot,t)|^2,~~\mbox{if}~~\frac{10}{9}< m\leq2,\\
 \disp\int_{\Omega}\left[n_{\varepsilon}(\cdot,t)^{m-1}(\cdot,t)+   c_{\varepsilon}^2(\cdot,t)+ | {u_{\varepsilon}}(\cdot,t)|^2\right],~~\mbox{if}~~m>2,\\
   \end{array}\right.\label{ghbbnnnnmmnjjffff1.1hhjjddssggtyy}
\end{equation}
 which is a
  new estimate of chemotaxis--Navier--stokes system   {\bf with rotation} (see Lemmata \ref{lemma630jklhhjj}--\ref{lemmaghjssddgghhmk4563025xxhjklojjkkk}),
   although, the part of \dref{ghbbnnnnmmnjjffff1.1hhjjddssggtyy} ($m<2$) has been used to solve the chemotaxis-(Navier)-Stokes system  {\bf without rotation} (see \cite{Bellomo1216,Lankeitffg11,Winklercvb12176}).
 Here $(n_\varepsilon,c_{\varepsilon},u_{\varepsilon})$ is the solution of the suitable approximate problem of \dref{1.1}. We guess that \dref{ghbbnnnnmmnjjffff1.1hhjjddssggtyy} can also be dealt with other types of systems,  e.g., Keller-Segel-Navier-Stokes system with nonlinear diffusion (see our recent paper \cite{Zheddfffggngddffsdsd6}).
 Then, in view of the estimates \dref{ghbbnnnnmmnjjffff1.1hhjjddssggtyy}, the suitable
interpolation arguments (see Lemma  \ref{drfe116lemma70hhjj}) and the basic a priori information  (see Lemma \ref{ghjssdeedrfe116lemma70hhjj}), one  can obtain boundedness of
\begin{equation}\|{n}_\varepsilon(\cdot, t)\|_{L^{p_0}(\Omega)}\leq C~~ \mbox{for all}~~ t\in(0, T_{max,\varepsilon}) ~~~\mbox{with}~~p_0>3
 \label{ssdffddfz2.571cghvvgjjllssdffll}
\end{equation}
and $C:=C(\varepsilon)$ depends on $\varepsilon$ (see Lemmata \ref{lemffggma456hjojjkkyhuissddff}--\ref{lemffggma456hjkhddff}).
With estimate of  \dref{ssdffddfz2.571cghvvgjjllssdffll} at hand, by using  variation-of-constants,  smoothing properties of the
Stokes semigroup and  Moser-type iteration,  we can show that our approximate solutions $(n_\varepsilon,c_{\varepsilon},u_{\varepsilon})$ are actually global
in time. Finally, by the interpolation inequality, we derive a priori estimates for the approximate solutions $(n_\varepsilon,c_{\varepsilon},u_{\varepsilon})$
to the approximate problems of problem \dref{1.1} and  complete the proof
of main results by an approximation procedure.

The rest of this paper is organized as follows. In the following section, we state our main results,  introduce the regularized system of \dref{1.1} and
collect  some basic estimates which will be useful for proofs later on. In Section 3, we derive  a series of useful estimates which depend on $\varepsilon$ and
then obtain the global existence of the regularized problems. In Section 4, in light of   the Gagliardo-Nirenberg inequality and the other some basic analysis,
we derive  some  $\varepsilon$-independent boundedness  of the time derivatives of certain powers of $n_{\varepsilon},c_{\varepsilon}$ and $u_{\varepsilon}$.
 In the final step, it is proved that \dref{1.1} possesses at least one weak  solution by the  Aubin--Lions lemma, the standard parabolic regularity theory and  the Egorov theorem.
\section{Preliminaries and  main results}

Due to hypothesis \dref{ghnjmk9161gyyhuug}, $\kappa\neq0$ and the presence of tensor-valued $S$ in system \dref{1.1}, we need to consider an appropriately regularized problem of \dref{1.1} at first. Indeed, following  the idea of  \cite{Winkler51215} (see also \cite{Tao71215,Zhangddffcvb12176}),
the corresponding regularized problem is introduced as follows:
\begin{equation}
 \left\{\begin{array}{ll}
   n_{\varepsilon t}+u_{\varepsilon}\cdot\nabla n_{\varepsilon}=\nabla\cdot(D_{\varepsilon}(n_{\varepsilon})\nabla n_{\varepsilon})-\nabla\cdot(n_{\varepsilon}F_{\varepsilon}(n_{\varepsilon})S_\varepsilon(x, n_{\varepsilon}, c_{\varepsilon})\cdot\nabla c_{\varepsilon}),\quad
x\in \Omega, t>0,\\
    c_{\varepsilon t}+u_{\varepsilon}\cdot\nabla c_{\varepsilon}=\Delta c_{\varepsilon}-n_{\varepsilon}c_{\varepsilon},\quad
x\in \Omega, t>0,\\
u_{\varepsilon t}+\nabla P_{\varepsilon}=\Delta u_{\varepsilon}-\kappa (Y_{\varepsilon}u_{\varepsilon} \cdot \nabla)u_{\varepsilon}+n_{\varepsilon}\nabla \phi ,\quad
x\in \Omega, t>0,\\
\nabla\cdot u_{\varepsilon}=0,\quad
x\in \Omega, t>0,\\
 \disp{\nabla n_{\varepsilon}\cdot\nu=\nabla c_{\varepsilon}\cdot\nu=0,u_{\varepsilon}=0,\quad
x\in \partial\Omega, t>0,}\\
\disp{n_{\varepsilon}(x,0)=n_0(x),c_{\varepsilon}(x,0)=c_0(x),u_{\varepsilon}(x,0)=u_0(x)},\quad
x\in \Omega,\\
 \end{array}\right.\label{1.1fghyuisda}
\end{equation}
where
\begin{equation}
 \begin{array}{ll}
 Y_{\varepsilon}w := (1 + \varepsilon A)^{-1}w ~~~~\mbox{for all}~~ w\in L^2_{\sigma}(\Omega)
 \end{array}\label{aasddffgg1.1fghyuisda}
\end{equation}
is the standard Yosida approximation,
%
%
$D_\varepsilon(s):=D_\varepsilon(s+\varepsilon)$, 
  \begin{equation}
 \begin{array}{ll}
S_\varepsilon(x, n, c) := \rho_\varepsilon(x)S(x, n, c),~~ x\in\bar{\Omega},~~n\geq0,~~c\geq0 ~~\mbox{and}~~\varepsilon\in(0, 1)
 \end{array}\label{3.10gghhjuuloollyuigghhhyy}
\end{equation}
and
\begin{equation}
F_{\varepsilon}(s)=\frac{1}{1+\varepsilon s}~~\mbox{for}~~s \geq 0.
\label{1.ffggvbbnxxccvvn1}
\end{equation}
Here $(\rho_\varepsilon)_{\varepsilon\in(0,1)} \in C^\infty_0 (\Omega)$
  be a family of standard cut-off functions satisfying $0\leq\rho_\varepsilon\leq 1$
   in $\Omega$
 and $\rho_\varepsilon\rightarrow1$ in $\Omega$
 as $\varepsilon\rightarrow0$.

 With the help of the Schauder fixed point theorem, the standard regularity theory of parabolic equations and the Stokes system,
in light of a
straightforward adaptation of a corresponding procedure in Lemma 2.1 of \cite{Winklercvb12176} to
the present setting (see also  Lemma 2.1 of \cite{Winkler11215}), we can easily  obtain following local
existence result of \dref{1.1fghyuisda}:
%
%
%
%
%
\begin{lemma}\label{lemma70}
Let $\Omega \subseteq \mathbb{R}^3 $ be a bounded  convex  domain with smooth boundary.
Assume
that
%
the initial data $(n_0,c_0,u_0)$ fulfills \dref{ccvvx1.731426677gg}.
Then there exist $T_{max,\varepsilon}\in  (0,\infty]$ and
a classical solution $(n_{\varepsilon}, c_{\varepsilon}, u_{\varepsilon}, P_{\varepsilon})$ of \dref{1.1fghyuisda} in
$\Omega\times(0, T_{max,\varepsilon})$ such that
\begin{equation}
 \left\{\begin{array}{ll}
 n_{\varepsilon}\in C^0(\bar{\Omega}\times[0,T_{max,\varepsilon}))\cap C^{2,1}(\bar{\Omega}\times(0,T_{max,\varepsilon})),\\
  c_{\varepsilon}{}\in  C^0(\bar{\Omega}\times[0,T_{max,\varepsilon}))\cap C^{2,1}(\bar{\Omega}\times(0,T_{max,\varepsilon})),\\
  u_{\varepsilon}{}\in  C^0(\bar{\Omega}\times[0,T_{max,\varepsilon}))\cap C^{2,1}(\bar{\Omega}\times(0,T_{max,\varepsilon})),\\
  P_{\varepsilon}{}\in  C^{1,0}(\bar{\Omega}\times(0,T_{max,\varepsilon})),\\
   \end{array}\right.\label{1.1ddfghyuisda}
\end{equation}
 classically solving \dref{1.1fghyuisda} in $\Omega\times[0,T_{max,\varepsilon})$.
%
Moreover, it holds that $n_\varepsilon$ and $c_\varepsilon$ are nonnegative in
$\Omega\times(0, T_{max,\varepsilon})$, and
\begin{equation}
\limsup_{t\nearrow T_{max,\varepsilon}}(\|n_\varepsilon(\cdot, t)\|_{L^\infty(\Omega)}+\|c_\varepsilon(\cdot, t)\|_{W^{1,\infty}(\Omega)}+\|A^\gamma u_\varepsilon(\cdot, t)\|_{L^{2}(\Omega)})=\infty,
\label{1.163072x}
\end{equation}
where $\gamma$ is given by \dref{ccvvx1.731426677gg}.
\end{lemma}

\begin{lemma}(\cite{Lankeitffg11})\label{lemmffffgga630jklhhjj}
 Let $w\in C^2(\bar{\Omega})$
  satisfy $\nabla w\cdot\nu  = 0$ on $\partial\Omega$.

    (i) Then
$$\frac{\partial|\nabla w|^2}{\partial\nu} \leq C_{\partial\Omega}|\nabla w|^2,$$
where $C_{\partial\Omega}$ is an upper bound on the curvature of $\partial\Omega$.

(ii) Furthermore, for any $\delta > 0$ there is $C(\delta) > 0$ such that every $w\in C^2(\bar{\Omega})$ with $\nabla w\cdot\nu  = 0$ on $\partial\Omega$ fulfils
$$\|w\|_{L^2(\partial\Omega)}\leq \delta \|\Delta w\|_{L^2(\Omega)} + C(\delta) \|w\|_{L^2(\Omega)} .$$

(iii) For any positive $w\in C^2(\bar{\Omega})$
 \begin{equation}\|\Delta w^{\frac{1}{2}}\|_{L^2(\Omega)}\leq \frac{1}{2} \|w^{\frac{1}{2}}\Delta \ln w\|_{L^2(\Omega)} +
 \frac{1}{4}  \|w^{-\frac{3}{2}}|\nabla w|^2\|_{L^2(\Omega)}.
\label{3.10gghhjuulofffollgghhhyhh}
\end{equation}

(iv) There are $C > 0$ and $\delta > 0$ such that every positive $w\in C^2(\bar{\Omega})$ fulfilling $\nabla w\cdot\nu  = 0$ on $\partial\Omega$ satisfies
\begin{equation}
 \begin{array}{rl}
  &\disp{-2\int_{\Omega}\frac{|\Delta  w|^2 }{ w}+\int_{\Omega}\frac{|\nabla  w|^2\Delta  w }{ w^2}\leq -\delta\int_{\Omega} w|D^2\ln w|^2-\delta\int_{\Omega}\frac{|\nabla  w|^4}{ w^3}+C\int_{\Omega} w.}\\
\end{array}\label{vcbbbbcvvgbhsvvbbsddacvvvvbbqwswddaassffssff3.10deerfgghhjuuloollgghhhyhh}
\end{equation}
\end{lemma}

\begin{lemma}\label{drfe116lemma70hhjj}(Lemma 3.8 of \cite{Winkler11215})
Let $q\geq1$,  \begin{equation}\lambda\in[2q+2,4q+1]
\label{3.10deerfgghhjuuloollgghhhyhh}
\end{equation}
and $\Omega\subset \mathbb{R}^3$ be a bounded  convex  domain with smooth boundary.
Then there exists $C > 0$ such that for all $\varphi\in C^2(\bar{\Omega})$ fulfilling $\varphi\cdot\frac{\partial\varphi}{\partial\nu}= 0$
 on $\partial\Omega$
 we have
 \begin{equation}
 \begin{array}{rl}
 &\|\nabla\varphi\|_{L^\lambda(\Omega)}\leq C\||\nabla\varphi|^{q-1}D^2\varphi\|_{L^2(\Omega)}^{\frac{2(\lambda-3)}{(2q-1)\lambda}}
 \|\varphi\|_{L^\infty(\Omega)}^{\frac{6q-\lambda}{(2q-1)\lambda}}+C\|\varphi\|_{L^\infty(\Omega)}.\\
\end{array}\label{aqwswddaassffssff3.10deerfgghhjuuloollgghhhyhh}
\end{equation}
\end{lemma}

Let us state two well-known results of solution of \dref{1.1fghyuisda}.

\begin{lemma}\label{ghjssdeedrfe116lemma70hhjj}
 The solution  of  \dref{1.1fghyuisda} satisfies
  \begin{equation}
 \begin{array}{rl}
 \|n_{\varepsilon}(\cdot,t)\|_{L^1(\Omega)}=\|n_0\|_{L^1(\Omega)}~~~\mbox{for all}~~t\in (0, T_{max,\varepsilon})
\end{array}\label{vgbhssddaqwswddaassffssff3.10deerfgghhjuuloollgghhhyhh}
\end{equation}
and
  \begin{equation}
 \begin{array}{rl}
 \|c_{\varepsilon}(\cdot,t)\|_{L^\infty(\Omega)}\leq\|c_0\|_{L^\infty(\Omega)}~~~\mbox{for all}~~t\in (0, T_{max,\varepsilon}).
\end{array}\label{hnjmssddaqwswddaassffssff3.10deerfgghhjuuloollgghhhyhh}
\end{equation}
\end{lemma}


\begin{lemma}\label{lemma630jklhhjccvvfggj}
Let $m>\frac{4}{3}$.
There exists $ C > 0$ such that
 the solution of \dref{1.1fghyuisda} satisfies
%
%
%
%
\begin{equation}
\frac{d}{dt}\int_{\Omega}{|u_{\varepsilon}|^2}+\int_{\Omega}{|\nabla u_{\varepsilon}|^2}\leq\frac{1}{8}\|\nabla n_{\varepsilon}^{m-1}\|_{L^{2}(\Omega)}^{2}+C~~\mbox{for all}~~ t\in(0, T_{max,\varepsilon}),
\label{ddz2.5ghjff48cfg924ghffggyuji}
\end{equation}
\end{lemma}
\begin{proof}Multiplying the
third equation of \dref{1.1fghyuisda} by $u_\varepsilon$, and then integrating by parts over $\Omega$ and using $\nabla\cdot u_{\varepsilon}=0$, it follows that
\begin{equation}
\frac{1}{2}\frac{d}{dt}\int_{\Omega}{|u_{\varepsilon}|^2}+\int_{\Omega}{|\nabla u_{\varepsilon}|^2}= \int_{\Omega}n_{\varepsilon}u_{\varepsilon}\cdot\nabla \phi~~\mbox{for all}~~ t\in(0, T_{max,\varepsilon}).
\label{ddddfgcz2.5ghju48cfg924ghyuji}
\end{equation}
Here we use the H\"{o}lder inequality, \dref{dd1.1fghyuisdakkkllljjjkk} and the continuity of the embedding $W^{1,2}(\Omega)\hookrightarrow L^6(\Omega)$ and  to
find $C_1 > 0$ such that
\begin{equation}
\begin{array}{rl}
\disp\int_{\Omega}n_{\varepsilon}u_{\varepsilon}\cdot\nabla \phi\leq&\disp{\|\nabla \phi\|_{L^\infty(\Omega)}\|n_{\varepsilon}\|_{L^{\frac{6}{5}}(\Omega)}\|\nabla u_{\varepsilon}\|_{L^{2}(\Omega)}}\\
\leq&\disp{C_1\|n_{\varepsilon}\|_{L^{\frac{6}{5}}(\Omega)}\|\nabla u_{\varepsilon}\|_{L^{2}(\Omega)}~~\mbox{for all}~~ t\in(0, T_{max,\varepsilon}),}\\
\end{array}
\label{ddddfgcz2.5ghju48cfg924ghyuji}
\end{equation}
According to the Gagliardo--Nirenberg inequality and \dref{vgbhssddaqwswddaassffssff3.10deerfgghhjuuloollgghhhyhh}, it is readily to see that
\begin{equation}
\begin{array}{rl}
\disp\|n_{\varepsilon}\|_{L^{\frac{6}{5}(\Omega)}}=&\disp{\|n_{\varepsilon}^{m-1}\|_{L^{\frac{6}{5(m-1)}}(\Omega)}^{\frac{1}{m-1}}}\\
\leq&\disp{C_2\|\nabla n_{\varepsilon}^{m-1}\|_{L^{2}(\Omega)}^{\frac{1}{6m-7}}\| n_{\varepsilon}^{m-1}\|_{L^{\frac{1}{m-1}}(\Omega)}^{\frac{1}{m-1}-\frac{1}{6m-7}}}\\
\leq&\disp{C_3(\|\nabla n_{\varepsilon}^{m-1}\|_{L^{2}(\Omega)}^{\frac{1}{6m-7}}+1)~~\mbox{for all}~~ t\in(0, T_{max,\varepsilon})}\\
\end{array}
\label{ddddfgcz2.5ghju4cddfff8cfg924gjjkkkhyuji}
\end{equation}
with some positive constants  $C_2$ and $C_3$ independent of $\varepsilon$.
Next, by   \dref{ddddfgcz2.5ghju48cfg924ghyuji}--\dref{ddddfgcz2.5ghju4cddfff8cfg924gjjkkkhyuji} and using the
Young inequality and $m>\frac{4}{3}$  yields
\begin{equation}
\begin{array}{rl}
\disp{\int_{\Omega}n_{\varepsilon}u_{\varepsilon}\cdot\nabla \phi}\leq &\disp{\frac{1}{2}\|\nabla u_{\varepsilon}\|_{L^{2}(\Omega)}^2+C_4(\|\nabla n_{\varepsilon}^{m-1}\|_{L^{2}(\Omega)}^{\frac{2}{6m-7}}+1)}\\
\leq &\disp{\frac{1}{2}\|\nabla u_{\varepsilon}\|_{L^{2}(\Omega)}^2+\frac{1}{8}\|\nabla n_{\varepsilon}^{m-1}\|_{L^{2}(\Omega)}^{2}+C_5~~\mbox{for all}~~ t\in(0, T_{max,\varepsilon}).}\\
\end{array}
\label{ddddfgcxccdd2.5ghju4cvvbbttthdfff8cfg924ghyuji}
\end{equation}
Here $C_4$ and $C_5$ are positive constants independent of $\varepsilon$.
Finally, putting \dref{ddddfgcxccdd2.5ghju4cvvbbttthdfff8cfg924ghyuji} into \dref{ddddfgcz2.5ghju48cfg924ghyuji}, one obtains \dref{ddz2.5ghjff48cfg924ghffggyuji}.
\end{proof}

\begin{lemma}\label{lemma630jklhhjj}
Let $m>\frac{2}{3}$.
There exists $ C > 0$ such that for every
$\delta_1>0$,
 the solution of \dref{1.1fghyuisda} satisfies
%
%
%
%
\begin{equation}
\frac{d}{dt}\int_{\Omega}{|u_{\varepsilon}|^2}+\int_{\Omega}{|\nabla u_{\varepsilon}|^2}\leq\delta_1\int_{\Omega}\frac{D_{\varepsilon}(n_{\varepsilon})|\nabla n_{\varepsilon}|^2}{n_{\varepsilon}}+C~~\mbox{for all}~~ t\in(0, T_{max,\varepsilon}),
\label{ddddfgcz2.5ghju48cfg924ghffggyuji}
\end{equation}
\end{lemma}
\begin{proof}
We begin with \dref{ddddfgcz2.5ghju48cfg924ghyuji},  the Gagliardo--Nirenberg inequality and \dref{vgbhssddaqwswddaassffssff3.10deerfgghhjuuloollgghhhyhh} ensure
\begin{equation}
\begin{array}{rl}
\disp\|n_{\varepsilon}\|_{L^{\frac{6}{5}(\Omega)}}=&\disp{\|n_{\varepsilon}^{\frac{m}{2}}\|_{L^{\frac{12}{5m}}(\Omega)}^{\frac{2}{m}}}\\
\leq&\disp{C_1\|\nabla n_{\varepsilon}^{\frac{m}{2}}\|_{L^{2}(\Omega)}^{\frac{1}{3m-1}}\| n_{\varepsilon}^{\frac{m}{2}}\|_{L^{\frac{2}{m}}(\Omega)}^{\frac{2}{m}-\frac{1}{3m-1}}}\\
\leq&\disp{C_2(\|\nabla n_{\varepsilon}^{\frac{m}{2}}\|_{L^{2}(\Omega)}^{\frac{1}{3m-1}}+1)~~\mbox{for all}~~ t\in(0, T_{max,\varepsilon}),}\\
\end{array}
\label{ddddfgcz2.5ghju4cddfff8cfg924ghyuji}
\end{equation}
where $C_1$ and $C_2$ are positive constants independent of $\varepsilon$.
Next, substituting \dref{ddddfgcz2.5ghju4cddfff8cfg924ghyuji} into \dref{ddddfgcz2.5ghju48cfg924ghyuji} and using the
Young inequality, \dref{ghnjmk9161gyyhuug} and $m>\frac{2}{3}$  yields
\begin{equation}
\begin{array}{rl}
\|n_{\varepsilon}\|_{L^{\frac{6}{5}(\Omega)}}\leq &\disp{\frac{1}{2}\|\nabla u_{\varepsilon}\|_{L^{2}(\Omega)}^2+C_3(\|\nabla n_{\varepsilon}^{\frac{m}{2}}\|_{L^{2}(\Omega)}^{\frac{2}{3m-1}}+1)}\\
\leq &\disp{\frac{1}{2}\|\nabla u_{\varepsilon}\|_{L^{2}(\Omega)}^2+\frac{\delta_1}{2C_D}\|\nabla n_{\varepsilon}^{\frac{m}{2}}\|_{L^{2}(\Omega)}^{2}+C_4}\\
\leq &\disp{\frac{1}{2}\|\nabla u_{\varepsilon}\|_{L^{2}(\Omega)}^2+\frac{\delta_1}{2}\int_{\Omega}\frac{D_{\varepsilon}(n_{\varepsilon})|\nabla n_{\varepsilon}|^2}{n_{\varepsilon}}+C_4~~\mbox{for all}~~ t\in(0, T_{max,\varepsilon}),}\\
\end{array}
\label{ddddfgczxxccdd2.5ghju4cddfff8cfg924ghyuji}
\end{equation}
where $C_3$ and $C_4$ are positive constants independent of $\varepsilon$.
Finally, collecting \dref{ddddfgcz2.5ghju48cfg924ghyuji} and \dref{ddddfgczxxccdd2.5ghju4cddfff8cfg924ghyuji}, we can conclude \dref{ddddfgcz2.5ghju48cfg924ghffggyuji}.
\end{proof}

\begin{lemma}\label{ghjsgghhsdeedrfe116lemma70hhjj} Let $\frac{10}{9}< m\leq2$.
 There exist $\mu_0, C > 0$ such that for every $\varepsilon > 0$ and $\delta_i(i=2,3,4,5)>0$
 \begin{equation}
 \begin{array}{rl}
 &\disp{\frac{d}{dt}\int_{\Omega}\frac{|\nabla c_{\varepsilon}|^2}{c_{\varepsilon}}+\mu_0\int_{\Omega}c_{\varepsilon}|D^2\ln c_{\varepsilon}|^2+(\mu_0-\frac{\delta_2}{4}-\frac{\delta_3}{4})\int_{\Omega}\frac{|\nabla c_{\varepsilon}|^4}{c_{\varepsilon}^3}}\\
 \leq&\disp{ (\frac{\delta_4}{4}+\frac{\delta_5}{4})\int_{\Omega}\frac{D_{\varepsilon}(n_{\varepsilon})|\nabla n_{\varepsilon}|^2}{n_{\varepsilon}}+\frac{4}{\delta_2}\|c_0\|_{L^\infty(\Omega)}\int_{\Omega}|\nabla u_{\varepsilon}|^2+C~~~\mbox{for all}~~t\in (0, T_{max,\varepsilon})}\\
\end{array}\label{vgbhsvvbbsddaqwswddaassffssff3.10deerfgghhjuuloollgghhhyhh}
\end{equation}
\end{lemma}
\begin{proof}
We begin by computing $\frac{d}{dt}\int_{\Omega}\frac{|\nabla c_{\varepsilon}|^2}{c_{\varepsilon}}$.
 For any 
 $t\in(0, T_{max,\varepsilon})$, we have
\begin{equation}
 \begin{array}{rl}
 \disp\frac{d}{dt}\disp\int_{\Omega}\frac{|\nabla c_{\varepsilon}|^2}{c_{\varepsilon}}=&\disp{2\int_{\Omega}\frac{\nabla c_{\varepsilon}\cdot\nabla c_{\varepsilon t}}{c_{\varepsilon}}-\int_{\Omega}\frac{|\nabla c_{\varepsilon}|^2c_{\varepsilon t}}{c_{\varepsilon}^2}}\\
 =&\disp{-2\int_{\Omega}\frac{\Delta c_{\varepsilon} c_{\varepsilon t}}{c_{\varepsilon}}+\int_{\Omega}\frac{|\nabla c_{\varepsilon}|^2c_{\varepsilon t}}{c_{\varepsilon}^2}}\\
  =&\disp{-2\int_{\Omega}\frac{|\Delta c_{\varepsilon}|^2 }{c_{\varepsilon}}+2\int_{\Omega}\frac{\Delta c_{\varepsilon} n_{\varepsilon} c_{\varepsilon}}{c_{\varepsilon}}+2\int_{\Omega}\frac{\Delta c_{\varepsilon}}{c_{\varepsilon}}u_{\varepsilon}\cdot\nabla c_{\varepsilon}}\\
  &+\disp{\int_{\Omega}\frac{|\nabla c_{\varepsilon}|^2\Delta c_{\varepsilon} }{c_{\varepsilon}^2}-\int_{\Omega}\frac{|\nabla c_{\varepsilon}|^2n_{\varepsilon} c_{\varepsilon} }{c_{\varepsilon}^2}-\int_{\Omega}\frac{|\nabla c_{\varepsilon}|^2u_{\varepsilon}\cdot\nabla c_{\varepsilon} }{c_{\varepsilon}^2}.}\\
\end{array}\label{vgbhsvvbbsddacvvvvbbqwswddaassffssff3.10deerfgghhjuuloollgghhhyhh}
\end{equation}
 From (vi) of  Lemma \ref{lemmffffgga630jklhhjj}, by Young inequality, there exist $\mu_0>0$ and $C(\mu_0)>0$
 such that
\begin{equation}
 \begin{array}{rl}
  &\disp{-2\int_{\Omega}\frac{|\Delta c_{\varepsilon}|^2 }{c_{\varepsilon}}+\int_{\Omega}\frac{|\nabla c_{\varepsilon}|^2\Delta c_{\varepsilon} }{c_{\varepsilon}^2}\leq -\mu_0\int_{\Omega}c_{\varepsilon}|D^2\ln c_{\varepsilon}|^2-\mu_0\int_{\Omega}\frac{|\nabla c_{\varepsilon}|^4}{c_{\varepsilon}^3}+C(\mu_0)\int_{\Omega}c_{\varepsilon}}\\
\end{array}\label{vccvvgbhsvvbbsddacvvvvbbqwswddaassffssff3.10deerfgghhjuuloollgghhhyhh}
\end{equation}
for all $t\in (0, T_{max,\varepsilon})$.
 As to the terms containing $u_{\varepsilon}$, we note that for all $\varepsilon > 0$
$$
 \begin{array}{rl}
&\disp{2\int_{\Omega}\frac{\Delta c_{\varepsilon}}{c_{\varepsilon}}(u_{\varepsilon}\cdot\nabla c_{\varepsilon})}\\
=&\disp{2\int_{\Omega}\frac{|\nabla c_{\varepsilon}|^2}{c_{\varepsilon}^2}u_{\varepsilon}\cdot\nabla c_{\varepsilon}-2\int_{\Omega}\frac{1}{c_{\varepsilon}}\nabla c_{\varepsilon}\cdot(\nabla u_{\varepsilon}\nabla c_{\varepsilon})-2\int_{\Omega}\frac{1}{c_{\varepsilon}}u_{\varepsilon}\cdot D^2c_{\varepsilon}\nabla c_{\varepsilon}~~\mbox{for all}~~t\in (0, T_{max,\varepsilon})}\\
\end{array}
$$
 and
 $$
 \begin{array}{rl}
\disp{\int_{\Omega}\frac{|\nabla c_{\varepsilon}|^2}{c_{\varepsilon}^2}u_{\varepsilon}\cdot\nabla c_{\varepsilon}=2\int_{\Omega}\frac{1}{c_{\varepsilon}}u_{\varepsilon}\cdot D^2c_{\varepsilon}\nabla c_{\varepsilon}~~\mbox{for all}~~t\in (0, T_{max,\varepsilon}),}\\
\end{array}
$$
so that due to the Young inequality and Lemma \ref{ghjssdeedrfe116lemma70hhjj}, for any $\delta_2 > 0 $
 \begin{equation}
 \begin{array}{rl}
\disp 2\int_{\Omega}\frac{\Delta c_{\varepsilon}}{c_{\varepsilon}}(u_{\varepsilon}\cdot\nabla c_{\varepsilon})-\int_{\Omega}\frac{|\nabla c_{\varepsilon}|^2}{c_{\varepsilon}^2}u_{\varepsilon}\cdot\nabla c_{\varepsilon}= &\disp{-2\int_{\Omega}\frac{1}{c_{\varepsilon}}\nabla c_{\varepsilon}\cdot(\nabla u_{\varepsilon}\nabla c_{\varepsilon})}\\
\leq &\disp{\frac{\delta_2}{4}\int_{\Omega}\frac{|\nabla c_{\varepsilon}|^4}{c_{\varepsilon}^3}+\frac{4}{\delta_2}\int_{\Omega}c_{\varepsilon}|\nabla u_{\varepsilon}|^2}\\
\leq &\disp{\frac{\delta_2}{4}\int_{\Omega}\frac{|\nabla c_{\varepsilon}|^4}{c_{\varepsilon}^3}+C_1\int_{\Omega}|\nabla u_{\varepsilon}|^2~~\mbox{for all}~~t\in (0, T_{max,\varepsilon})}\\
\end{array}\label{vgbhccvsvvbbsddacvvvvbbqwswddvvbbaassffssff3.10deerfgghhjuuloollgghhhyhh}
\end{equation}
with some $C_1:=\frac{4}{\delta_2}\|c_0\|_{L^\infty(\Omega)}$. And an integration by parts, the Young inequality, \dref{ghnjmk9161gyyhuug} and \dref{hnjmssddaqwswddaassffssff3.10deerfgghhjuuloollgghhhyhh}  shows
\begin{equation}
 \begin{array}{rl}
 2\disp\int_{\Omega}\frac{\Delta c_{\varepsilon} n_{\varepsilon} c_{\varepsilon}}{c_{\varepsilon}}=&\disp{-2\int_{\Omega}\nabla n_{\varepsilon}\cdot \nabla c_{\varepsilon}}\\
 \leq&\disp{\frac{\delta_3}{4}\int_{\Omega}\frac{|\nabla c_{\varepsilon}|^4}{c_{\varepsilon}^3}+2^{\frac{4}{3}}\delta^{-\frac{1}{3}}_3\int_{\Omega}c_{\varepsilon}|\nabla n_{\varepsilon}|^{\frac{4}{3}}}\\
\leq&\disp{\frac{\delta_3}{4}\int_{\Omega}\frac{|\nabla c_{\varepsilon}|^4}{c_{\varepsilon}^3}+\frac{\delta_4}{4C_{D}}\int_{\Omega}n_{\varepsilon}^{m-2}|\nabla n_{\varepsilon}|^{2}+C_2\int_{\Omega}n_{\varepsilon}^{4-2m}c_{\varepsilon}^3}\\
\leq&\disp{\frac{\delta_3}{4}\int_{\Omega}\frac{|\nabla c_{\varepsilon}|^4}{c_{\varepsilon}^3}+\frac{\delta_4}{4}\int_{\Omega}\frac{D_{\varepsilon}(n_{\varepsilon})|\nabla n_{\varepsilon}|^2}{n_{\varepsilon}}+C_3\int_{\Omega}n_{\varepsilon}^{4-2m}~~
\mbox{for all}~~t\in (0, T_{max,\varepsilon}),}\\
\end{array}\label{vgbhsvvbbsddacvvvvcvvvbbqwswddaassffssff3.10deerfgghhjuuloollgghhhyhh}
\end{equation}
where $\delta_3,\delta_4,C_2:=C_2(\delta_3,\delta_4), C_3:=C_3(\delta_3,\delta_4,\|c_0\|_{L^\infty(\Omega)})$  are positive constants.

Case $\frac{10}{9}<m<\frac{3}{2}$:
It is easy to deduce from the Gagliardo--Nirenberg inequality and \dref{vgbhssddaqwswddaassffssff3.10deerfgghhjuuloollgghhhyhh} that
\begin{equation}
\begin{array}{rl}
C_3\disp\int_{\Omega}n_{\varepsilon}^{4-2m}=&\disp{C_3\| n_{\varepsilon}^{\frac{m}{2}}\|_{L^{\frac{2(4-2m)}{m}}(\Omega)}^{\frac{2(4-2m)}{m}}}\\
\leq &\disp{C_4\|\nabla n_{\varepsilon}^{\frac{m}{2}}\|_{L^{2}(\Omega)}^{\frac{2(4-2m)\mu_1}{m}}\| n_{\varepsilon}^{\frac{m}{2}}\|_{L^{\frac{2}{m}}(\Omega)}^{\frac{2(4-2m)(1-\mu_1)}{m}}+\| n_{\varepsilon}^{\frac{m}{2}}\|_{L^{\frac{2}{m}}(\Omega)}^{\frac{2(4-2m)}{m}}}\\
\leq &\disp{C_5(\|\nabla n_{\varepsilon}^{\frac{m}{2}}\|_{L^{2}(\Omega)}^{\frac{2(4-2m)\mu_1}{m}}+1)}\\
= &\disp{C_5(\|\nabla n_{\varepsilon}^{\frac{m}{2}}\|_{L^{2}(\Omega)}^{\frac{6(3-2m)}{3m-1}}+1)~~\mbox{for all}~~t\in (0, T_{max,\varepsilon}),}\\
\end{array}\label{vgbhsvvbbsddffffggdaccvvcvvvvcvvvbbqwswddaassffssff3.10deerfgghhjuuloollgghhhyhh}
\end{equation}
where $C_4$ and $C_5$ are positive constants,
$$\mu_1=\frac{\frac{3m}{2}-\frac{3m}{2(4-2m)}}{\frac{3m-1}{2}}\in(0,1).$$
Now, in view of $m>\frac{10}{9},$ with the help of the Young inequality and \dref{vgbhsvvbbsddffffggdaccvvcvvvvcvvvbbqwswddaassffssff3.10deerfgghhjuuloollgghhhyhh}, for any $\delta_5>0$, we have
\begin{equation}
\begin{array}{rl}
&\disp{C_3\int_{\Omega}n_{\varepsilon}^{4-2m}\leq\frac{\delta_5}{4C_D}\|\nabla n_{\varepsilon}^{\frac{m}{2}}\|_{L^{2}(\Omega)}^{2}+C_6~~\mbox{for all}~~t\in (0, T_{max,\varepsilon})}\\
\end{array}\label{vgbhccvvvvsvvbbsddffffggdaccvvcvvvvcvvvbbqwswddaassffssff3.10deerfgghhjuuloollgghhhyhh}
\end{equation}
with some $C_6>0$.

Case $\frac{3}{2}\leq m\leq2$: With the help of the Young inequality and \dref{vgbhssddaqwswddaassffssff3.10deerfgghhjuuloollgghhhyhh}, we derive that
\begin{equation}
\begin{array}{rl}
C_3\disp\int_{\Omega}n_{\varepsilon}^{4-2m}
\leq &\disp{\int_{\Omega}n_{\varepsilon}+C_7}\\
\leq &\disp{C_8~~\mbox{for all}~~t\in (0, T_{max,\varepsilon}),}\\
\end{array}\label{vgbhsvvbbsdffssff3.10deerfgghhjuuloollgghhhyhh}
\end{equation}
where $C_7$ and $C_8$ are positive constants independent of $\varepsilon.$
Finally, we utilize \dref{vccvvgbhsvvbbsddacvvvvbbqwswddaassffssff3.10deerfgghhjuuloollgghhhyhh}--\dref{vgbhsvvbbsdffssff3.10deerfgghhjuuloollgghhhyhh}
 and \dref{vgbhsvvbbsddacvvvvbbqwswddaassffssff3.10deerfgghhjuuloollgghhhyhh} to deduce  the results.
\end{proof}

\begin{lemma}\label{ghjhhjssjkkllsgghhsdeedrfe116lemma70hhjj}  Let $\frac{10}{9}< m\leq2$ and $\delta>0$.
There is $C > 0$ such that for any 
$\delta_6$ and  $\delta_7$
 \begin{equation}
 \begin{array}{rl}
 \disp\frac{d}{dt}\int_{\Omega}n_{\varepsilon} \ln n_{\varepsilon}+
(1-\frac{\delta_7}{4})\int_{\Omega}\frac{D_{\varepsilon}(n_{\varepsilon})|\nabla n_{\varepsilon}|^2}{n_{\varepsilon}}\leq \frac{\delta_6}{4}\int_{\Omega}\frac{|\nabla c_{\varepsilon}|^4}{c_{\varepsilon}^3} +C~\mbox{for all}~t\in (0, T_{max,\varepsilon}).
\end{array}\label{vgccvssbbbbhsvvbbsddaqwswddaassffssff3.10deerfgghhjuuloollgghhhyhh}
\end{equation}
\end{lemma}
\begin{proof}
Firstly, using the first equation of \dref{1.1fghyuisda} and \dref{1.ffggvbbnxxccvvn1}, from integration by parts we obtain from \dref{x1.73142vghf48gg}
%
%
%
\begin{equation}
 \begin{array}{rl}
\frac{d}{dt}&\disp\int_{\Omega}n_{\varepsilon} \ln n_{\varepsilon}\\
 =&\disp{\int_{\Omega}n_{\varepsilon t} \ln n_{\varepsilon}+
\int_{\Omega}n_{\varepsilon t}}\\
=&\disp{\int_{\Omega}\nabla\cdot(D_{\varepsilon}(n_{\varepsilon})\nabla n_{\varepsilon} ) \ln n_{\varepsilon}-
\int_{\Omega}\ln n_{\varepsilon}\nabla\cdot(n_{\varepsilon}F_{\varepsilon}(n_{\varepsilon})S_\varepsilon(x, n_{\varepsilon}, c_{\varepsilon})\cdot\nabla c_{\varepsilon})-\int_{\Omega}\ln n_{\varepsilon}u_{\varepsilon}\cdot\nabla n_{\varepsilon}}\\
\leq&\disp{-\int_{\Omega}\frac{D_{\varepsilon}(n_{\varepsilon})|\nabla n_{\varepsilon}|^2}{n_{\varepsilon}}+
\int_{\Omega} S_0(c_{\varepsilon})|\nabla n_{\varepsilon}||\nabla c_{\varepsilon}|}\\
\end{array}\label{vgccvsckkcvvsbbbbhsvvbbsddaqwswddaassffssff3.10deerfgghhjuuloollgghhhyhh}
\end{equation}
for all $t\in (0, T_{max,\varepsilon})$. 
Now, in view of  \dref{hnjmssddaqwswddaassffssff3.10deerfgghhjuuloollgghhhyhh}, employing the same argument of \dref{vgbhsvvbbsddacvvvvcvvvbbqwswddaassffssff3.10deerfgghhjuuloollgghhhyhh}--\dref{vgbhsvvbbsdffssff3.10deerfgghhjuuloollgghhhyhh}, for any $\delta_6>0$ and $\delta_7>0$, we conclude that
\begin{equation}
 \begin{array}{rl}
 &\disp{\int_{\Omega} S_0(c_{\varepsilon})|\nabla n_{\varepsilon}||\nabla c_{\varepsilon}|}\\
\leq &\disp{S_0(\|c_0\|_{L^\infty(\Omega)})\int_{\Omega}|\nabla n_{\varepsilon}| |\nabla c_{\varepsilon}|}\\
 \leq&\disp{\frac{\delta_6}{4}\int_{\Omega}\frac{|\nabla c_{\varepsilon}|^4}{c_{\varepsilon}^3}+\frac{\delta_7}{4}\int_{\Omega}\frac{D_{\varepsilon}(n_{\varepsilon})|\nabla n_{\varepsilon}|^2}{n_{\varepsilon}}+C_1~~
\mbox{for all}~~t\in (0, T_{max,\varepsilon})}\\
\end{array}\label{vgbhsvvbbsddacvvvccvcvvvbbqwswddaassffssff3.10deerfgghhjuuloollgghhhyhh}
\end{equation}
with $C_1>0$ independent of $\varepsilon.$
Now, in conjunction with \dref{vgccvsckkcvvsbbbbhsvvbbsddaqwswddaassffssff3.10deerfgghhjuuloollgghhhyhh} and \dref{vgbhsvvbbsddacvvvccvcvvvbbqwswddaassffssff3.10deerfgghhjuuloollgghhhyhh}, we get the results.
This completes the proof of Lemma \ref{ghjhhjssjkkllsgghhsdeedrfe116lemma70hhjj}.
\end{proof}
Properly combining Lemmata \ref{lemma630jklhhjj}--\ref{ghjhhjssjkkllsgghhsdeedrfe116lemma70hhjj}, we arrive at the following Lemma, which plays a key rule in obtaining
 the existence of solutions to \dref{1.1fghyuisda}.
\begin{lemma}\label{lemmakkllgg4563025xxhjklojjkkk}
Let $\frac{10}{9}< m\leq2$ and  $S$ satisfy  \dref{x1.73142vghf48rtgyhu}--\dref{x1.73142vghf48gg}.
Suppose that \dref{ghnjmk9161gyyhuug} and \dref{dd1.1fghyuisdakkkllljjjkk}--\dref{ccvvx1.731426677gg}
holds.
Then there exists $C>0$ independent of $\varepsilon$ such that the solution of \dref{1.1fghyuisda} satisfies
\begin{equation}
\begin{array}{rl}
&\disp{\int_{\Omega}n_{\varepsilon}\ln n_{\varepsilon}+\int_{\Omega}|\nabla\sqrt{c_{\varepsilon}}|^2+\int_{\Omega}|u_{\varepsilon}|^2\leq C}\\
\end{array}
\label{czfvgb2.5ghhjuyuiihjj}
\end{equation}
for all $t\in(0, T_{max,\varepsilon})$.
Moreover,
for each $T\in(0, T_{max,\varepsilon})$, one can find a constant $C > 0$ independent of $\varepsilon$ such that
\begin{equation}
\begin{array}{rl}
&\disp{\int_{0}^T\int_{\Omega}  n_{\varepsilon}^{m-2} |\nabla {n_{\varepsilon}}|^2\leq C,}\\
\end{array}
\label{bnmbncz2.5ghhjuyuiihjj}
\end{equation}
\begin{equation}
\begin{array}{rl}
&\disp{\int_{0}^T\int_{\Omega} |\nabla {u_{\varepsilon}}|^2\leq C}\\
\end{array}
\label{bnmbncz2.5ghffghhhjuyuiihjj}
\end{equation}
and
\begin{equation}
\begin{array}{rl}
&\disp{\int_{0}^T\int_{\Omega} |\nabla {c_{\varepsilon}}|^4\leq C}\\
\end{array}
\label{vvcz2.5ghhjuyuiihjj}
\end{equation}
as well as
\begin{equation}
\begin{array}{rl}
&\disp{\int_{0}^T\int_{\Omega}c_{\varepsilon}|D^2\ln c_{\varepsilon}|^2\leq C.}\\
\end{array}
\label{cvffvbgvvcz2.5ghhjuyuiihjj}
\end{equation}
\end{lemma}
\begin{proof}
Take an evident linear combination of the inequalities provided by Lemmata \ref{lemma630jklhhjj}--\ref{ghjhhjssjkkllsgghhsdeedrfe116lemma70hhjj}, we conclude
\begin{equation}
 \begin{array}{rl}
 &\disp{\frac{d}{dt}\left(\int_{\Omega}\frac{|\nabla c_{\varepsilon}|^2}{c_{\varepsilon}}+L\int_{\Omega}n_{\varepsilon} \ln n_{\varepsilon}+K\int_{\Omega}|u_{\varepsilon}|^2\right)+(K-\frac{4}{\delta_2}\|c_0\|_{L^\infty(\Omega)})\int_{\Omega}|\nabla u_{\varepsilon}|^2+\mu_0\int_{\Omega}c_{\varepsilon}|D^2\ln c_{\varepsilon}|^2}\\
 &\disp{+[(\mu_0-\frac{\delta_2}{4}-\frac{\delta_3}{4})-L\frac{\delta_6}{4}]\int_{\Omega}\frac{|\nabla c_{\varepsilon}|^4}{c_{\varepsilon}^3}
 +[L(1-\frac{\delta_7}{4})-\frac{\delta_4}{4}-\frac{\delta_5}{4}-K\delta_1]\int_{\Omega}\frac{D_{\varepsilon}(n_{\varepsilon})|\nabla n_{\varepsilon}|^2}{n_{\varepsilon}}}\\
 \leq&\disp{ C~~~\mbox{for all}~~t\in (0, T_{max,\varepsilon}),}\\
\end{array}\label{vgbccvbbffeerfgghhjuuloollgghhhyhh}
\end{equation}
where $K,L$ are positive constants. Now, choosing $\delta_7=1,$ $\delta_6=\frac{\mu_0}{L},\delta_3=\mu_0,\delta_4=\delta_5=L,\delta_1=\frac{L}{8K}$ and
$\delta_2=\frac{8}{K}\|c_0\|_{L^\infty(\Omega)}$  and $K$ large enough such that $\frac{8}{K}\|c_0\|_{L^\infty(\Omega)}<\mu_0$
in \dref{vgbccvbbffeerfgghhjuuloollgghhhyhh}, one may arrive at
 \begin{equation}
 \begin{array}{rl}
 &\disp{\frac{d}{dt}\left(\int_{\Omega}\frac{|\nabla c_{\varepsilon}|^2}{c_{\varepsilon}}+L\int_{\Omega}n_{\varepsilon} \ln n_{\varepsilon}+K\int_{\Omega}|u_{\varepsilon}|^2\right)+\frac{K}{2}\int_{\Omega}|\nabla u_{\varepsilon}|^2+\mu_0\int_{\Omega}c_{\varepsilon}|D^2\ln c_{\varepsilon}|^2}\\
 &\disp{+\frac{\mu_0}{4}\int_{\Omega}\frac{|\nabla c_{\varepsilon}|^4}{c_{\varepsilon}^3}
 +\frac{L}{8}\int_{\Omega}\frac{D_{\varepsilon}(n_{\varepsilon})|\nabla n_{\varepsilon}|^2}{n_{\varepsilon}}}\\
 \leq&\disp{ C~~~\mbox{for all}~~t\in (0, T_{max,\varepsilon}).}\\
\end{array}\label{vgbccvbbffeerfgghcchjuuloollgghhhyhh}
\end{equation}
As a result,  we immediately obtain \dref{czfvgb2.5ghhjuyuiihjj}--\dref{cvffvbgvvcz2.5ghhjuyuiihjj} after integrating \dref{vgbccvbbffeerfgghcchjuuloollgghhhyhh}
 over $(0, T)$.
\end{proof}
In what follows, we are in a position to discuss the case $m>2$, we first give the following Lemma which plays a key rule in obtaining  the existence of  solution to main results.

\begin{remark}
Due to the strong nonlinear term $(u \cdot \nabla)u$, the methods used in \cite{Zhengssssssssdefr23} to derive the
higher-order estimates on the solutions $(n_{\varepsilon}, c_{\varepsilon}, u_{\varepsilon})$, which guarantee the solutions obtained are
indeed a locally bounded one, cannot be applied any more. To overcome this
difficulty, we need some new careful  analysis.
\end{remark}

\begin{lemma}\label{lemmaghjssddgghhmk4563025xxhjklojjkkk}
Let $m>2$.
Then there exists $C>0$ independent of $\varepsilon$ such that the solution of \dref{1.1fghyuisda} satisfies
\begin{equation}
\begin{array}{rl}
&\disp{\int_{\Omega}(n_{\varepsilon}+\varepsilon)^{m-1}+\int_{\Omega}   c_{\varepsilon}^2+\int_{\Omega}  | {u_{\varepsilon}}|^2\leq C~~~\mbox{for all}~~ t\in (0, T_{max,\varepsilon}).}\\
\end{array}
\label{czfvgb2.5ghhjuyuccvviihjj}
\end{equation}
In addition,
for each $T\in(0, T_{max,\varepsilon})$, one can find a constant $C > 0$ independent of $\varepsilon$ such that
\begin{equation}
\begin{array}{rl}
&\disp{\int_{0}^T\int_{\Omega} \left[ (D_{\varepsilon}(n_{\varepsilon}))^{\frac{2m-4}{m-1}}|\nabla n_{\varepsilon}|^2+(n_{\varepsilon}+\varepsilon)^{2m-4} |\nabla {n_{\varepsilon}}|^2+ |\nabla {c_{\varepsilon}}|^2+ |\nabla {u_{\varepsilon}}|^2\right]\leq C.}\\
\end{array}
\label{bnmbncz2.5ghhjuyuivvbnnihjj}
\end{equation}
\end{lemma}
\begin{proof}
Multiplying ${c_{\varepsilon}}$ on both sides of the second  equation of \dref{1.1fghyuisda} and using $\nabla\cdot u_\varepsilon=0$, one has after integration by
parts   that
\begin{equation}
\begin{array}{rl}
&\disp{\frac{1}{{2}}\frac{d}{dt}\|{c_{\varepsilon}}\|^{{{2}}}_{L^{{2}}(\Omega)}+
\int_{\Omega} |\nabla c_{\varepsilon}|^2 =-\int_{\Omega} n_{\varepsilon}c^2_{\varepsilon},}
\end{array}
\label{hhxxcdfvvjjcz2.5}
\end{equation}
which together with $n_{\varepsilon}\geq0$,  $c_{\varepsilon}\geq0$  and the Gronwall inequality implies that
\begin{equation}
\begin{array}{rl}
&\disp{\int_{\Omega}   c_{\varepsilon}^2\leq C_1~~~\mbox{for all}~~ t\in (0, T_{max,\varepsilon})}\\
\end{array}
\label{czfvgb2.5ghhjuyuccvvixxcvvihjj}
\end{equation}
and
\begin{equation}
\begin{array}{rl}
&\disp{\int_{0}^T\int_{\Omega}  |\nabla {c_{\varepsilon}}|^2\leq C_1~~\mbox{for all}~~ T\in (0, T_{max,\varepsilon})}\\
\end{array}
\label{bnmbncz2.5ghhjuyuivvbddfghhnnihjj}
\end{equation}
with some positive constant $C_1$.
Next, multiply the first equation in $\dref{1.1fghyuisda}$ by $({n_{\varepsilon}}+\varepsilon)^{p-1}$
 and combining with the second equation and using $\nabla\cdot u_\varepsilon=0$ and \dref{1.ffggvbbnxxccvvn1} implies that
\begin{equation}
\begin{array}{rl}
&\disp{\frac{1}{{p}}\frac{d}{dt}\|{n_{\varepsilon}}+\varepsilon\|^{{{p}}}_{L^{{p}}(\Omega)}+
\frac{C_D(p-1)}{2}\int_{\Omega}({n_{\varepsilon}}+\varepsilon)^{m+p-3} |\nabla n_{\varepsilon}|^2}\\
 \leq&\disp{\frac{(p-1)[S_0(\|c_0\|_{L^\infty(\Omega)})]^2}{2C_D}\int_\Omega ({n_{\varepsilon}}+\varepsilon)^{p+1-m}|\nabla {c}_{\varepsilon}|^2.}
\end{array}
\label{hhjjcz2.5}
\end{equation}
Now, choosing $p=m-1$ in \dref{hhjjcz2.5} and using \dref{bnmbncz2.5ghhjuyuivvbddfghhnnihjj} yields to
\begin{equation}
\begin{array}{rl}
&\disp{\int_{\Omega}   (n_{\varepsilon}+\varepsilon)^{m-1}\leq C_2~~~\mbox{for all}~~ t\in (0, T_{max,\varepsilon})}\\
\end{array}
\label{czfvgb2ghhjuyucvixxcvvihjj}
\end{equation}
and
\begin{equation}
\begin{array}{rl}
&\disp{\int_{0}^T\int_{\Omega}  (n_{\varepsilon}+\varepsilon)^{2m-4} |\nabla {n_{\varepsilon}}|^2\leq C_2~~~\mbox{for all}~~ T\in (0, T_{max,\varepsilon})}\\
\end{array}
\label{bnmbnc.5ghhjuyuivvbddhhjj}
\end{equation}
and some positive  constant $C_2$.
Next,  collecting \dref{ddz2.5ghjff48cfg924ghffggyuji} and \dref{bnmbnc.5ghhjuyuivvbddhhjj} and with some basic calculation, we conclude that there exists a positive constant $C_3$ such that
\begin{equation}
\begin{array}{rl}
&\disp{\int_{\Omega}  | {u_{\varepsilon}}|^2\leq C_3~~~\mbox{for all}~~ t\in (0, T_{max,\varepsilon})}\\
\end{array}
\label{czf2.5ghciihjj}
\end{equation}
and
\begin{equation}
\begin{array}{rl}
&\disp{\int_{0}^T\int_{\Omega}  |\nabla {u_{\varepsilon}}|^2\leq C_3.}\\
\end{array}
\label{mbn2.5ghhjuivnnihjj}
\end{equation}
Finally,  combining   \dref{czfvgb2.5ghhjuyuccvvixxcvvihjj}--\dref{bnmbncz2.5ghhjuyuivvbddfghhnnihjj} and \dref{czfvgb2ghhjuyucvixxcvvihjj}--\dref{mbn2.5ghhjuivnnihjj} and using \dref{ghnjmk9161gyyhuug}, we can get \dref{czfvgb2.5ghhjuyuccvviihjj} and \dref{bnmbncz2.5ghhjuyuivvbnnihjj}.
\end{proof}

\section{Global existence of the regularized problems}
To prove global existence of the regularized problems \dref{1.1fghyuisda}, whose proof will be postponed to the end of this subsection, we need to give a series of useful estimates.
For notational convenience, throughout this section we denote by $C$ or $C_i (i = 1, 2,\ldots)$ the generic positive constants which may depend on $\varepsilon$.
To this end, we intend to supplement Lemmata  \ref{lemmakkllgg4563025xxhjklojjkkk}--\ref{lemmaghjssddgghhmk4563025xxhjklojjkkk}
with bounds on $n_{\varepsilon}$. This will be the purpose of the following  lemmata.

\begin{lemma}\label{lemffggma456hjojjkkyhuissddff}
Assuming that $\frac{10}{9}<m\leq 2$ and $T_{max,\varepsilon}<+\infty.$
Then there exists a positive constant $C$ depends on $\varepsilon$ such that 
\begin{equation}\|{n}_\varepsilon(\cdot, t)\|_{L^p(\Omega)}\leq C~~ \mbox{for all}~~ t\in(0, T_{max,\varepsilon})
 \label{ssdffddfz2.571cghvvkbnnhggjjllll}
\end{equation}
with $p<\frac{38}{9}.$
\end{lemma}
\begin{proof}
Multiplying the first equation in \dref{hhjjcz2.5} by $n_{\varepsilon}^{p-1}$ with $p\in [m + 1, 2(m + 1)]$ and using
integration by parts  and the Young inequality, we obtain
\begin{equation}
\begin{array}{rl}
&\disp\frac{1}{{p}}\disp\frac{d}{dt}\|{n_{\varepsilon}}\|^{{{p}}}_{L^{{p}}(\Omega)}+
\disp C_D(p-1)\disp\int_{\Omega}n_{\varepsilon}^{m+p-3} |\nabla n_{\varepsilon}|^2\\
=&\disp{\int_\Omega -\nabla\cdot(n_{\varepsilon}F_{\varepsilon}(n_{\varepsilon})S_\varepsilon(x, n_{\varepsilon}, c_{\varepsilon})\cdot\nabla c_{\varepsilon})n_{\varepsilon}^{p-1}}\\
=&\disp{(p-1)\int_\Omega n_{\varepsilon}^{p-1}F_{\varepsilon}(n_{\varepsilon})S_\varepsilon(x, n_{\varepsilon}, c_{\varepsilon})\nabla n_{\varepsilon}\cdot\nabla c_{\varepsilon}}\\
\leq&\disp{\frac{(p-1)S_0(\|c_0\|_{L^\infty(\Omega)})}{\varepsilon}\int_\Omega n_{\varepsilon}^{p-2}|\nabla n_{\varepsilon}||\nabla c_{\varepsilon}|}\\
\leq&\disp{\frac{C_D(p-1)}{2}\disp\int_{\Omega}n_{\varepsilon}^{m+p-3} |\nabla n_{\varepsilon}|^2+\int_\Omega n_{\varepsilon}^{2({p-1-m}) }+C_1\int_\Omega|\nabla {c}_{\varepsilon}|^4}\\
\leq&\disp{\frac{C_D(p-1)}{2}\disp\int_{\Omega}n_{\varepsilon}^{m+p-3} |\nabla n_{\varepsilon}|^2+\int_\Omega n_{\varepsilon}^{p}+C_1\int_\Omega|\nabla {c}_{\varepsilon}|^4+C_2~~ \mbox{for all}~~ t\in(0, T_{max,\varepsilon})}\\
\end{array}
\label{hhjjccv55677cz2.5}
\end{equation}
and some positive constants $C_1$ and $C_2.$
Finally, we obtain \dref{ssdffddfz2.571cghvvkbnnhggjjllll} after by using \dref{vvcz2.5ghhjuyuiihjj} and  the Gronwall inequality.
The proof of Lemma \ref{lemffggma456hjojjkkyhuissddff} is completed.
\end{proof}

\begin{lemma}\label{lemffggddrftgyma456hjojjkkyhuissddff}
Suppose that $m>2$ and $T_{max,\varepsilon}<+\infty.$
Then there exists a positive constant $C$ depends on $\varepsilon$ such that 
\begin{equation}\|{n}_\varepsilon(\cdot, t)\|_{L^{m+1}(\Omega)}\leq C~~ \mbox{for all}~~ t\in(0, T_{max,\varepsilon}).
 \label{ssdffddfz2.571ssedrrcghvvkbnnhggjjllll}
\end{equation}
\end{lemma}
\begin{proof}
Multiplying the first equation in \dref{hhjjcz2.5} by $n_{\varepsilon}^{m}$, and
integrating them by parts over $\Omega$, one easily deduces from the Young inequality that there exists a positive constant $C_1$ such that
\begin{equation}
\begin{array}{rl}
&\disp\frac{1}{{m+1}}\disp\frac{d}{dt}\|{n_{\varepsilon}}\|^{{{m+1}}}_{L^{{m+1}}(\Omega)}+
\disp C_Dm\disp\int_{\Omega}n_{\varepsilon}^{2m-2} |\nabla n_{\varepsilon}|^2\\
=&\disp{\int_\Omega -\nabla\cdot(n_{\varepsilon}F_{\varepsilon}(n_{\varepsilon})S_\varepsilon(x, n_{\varepsilon}, c_{\varepsilon})\cdot\nabla c_{\varepsilon})n_{\varepsilon}^{m}}\\
=&\disp{m\int_\Omega n_{\varepsilon}^{m}F_{\varepsilon}(n_{\varepsilon})S_\varepsilon(x, n_{\varepsilon}, c_{\varepsilon})\nabla n_{\varepsilon}\cdot\nabla c_{\varepsilon}}\\
\leq&\disp{\frac{mS_0(\|c_0\|_{L^\infty(\Omega)})}{\varepsilon}\int_\Omega n_{\varepsilon}^{m-1}|\nabla n_{\varepsilon}||\nabla c_{\varepsilon}|}\\
\leq&\disp{\frac{C_Dm}{2}\disp\int_{\Omega}n_{\varepsilon}^{2m-2} |\nabla n_{\varepsilon}|^2+C_1\int_\Omega|\nabla {c}_{\varepsilon}|^2~~
\mbox{for all}~~ t\in(0, T_{max,\varepsilon}).}\\
\end{array}
\label{hhjjccv55ssd677cz2ssdeerr.5}
\end{equation}
Thus, in view of \dref{vvcz2.5ghhjuyuiihjj}, an application
of the Gronwall inequality immediately leads to \dref{ssdffddfz2.571ssedrrcghvvkbnnhggjjllll}.
\end{proof}

Properly combining Lemmata \ref{lemffggma456hjojjkkyhuissddff}--\ref{lemffggddrftgyma456hjojjkkyhuissddff}, we arrive at the following.
\begin{lemma}\label{lemffggma456hjkhddff}
Assuming that $m>\frac{10}{9}$ and $T_{max,\varepsilon}<+\infty.$
Then there exits a positive constant $C$ such that 
\begin{equation}\|{n}_\varepsilon(\cdot, t)\|_{L^{p_0}(\Omega)}\leq C~~ \mbox{for all}~~ t\in(0, T_{max,\varepsilon}) ~~~\mbox{with}~~p_0>3.
 \label{ssdffddfz2.571cghvvgjjllll}
\end{equation}
\end{lemma}
\begin{proof}If  $\frac{10}{9}<m\leq 2$, by Lemma \ref{lemffggma456hjojjkkyhuissddff}, we obtain that there exists a positive constant $C_1$ such that
\begin{equation}\|{n}_\varepsilon(\cdot, t)\|_{L^p(\Omega)}\leq C_1~~ \mbox{for all}~~ t\in(0, T_{max,\varepsilon})
 \label{ssdffddfz2.571cghvvkbnnvvtthyybbhggllll}
\end{equation}
with $p<\frac{38}{9}.$
While, if $m>2$, then by Lemma \ref{lemffggddrftgyma456hjojjkkyhuissddff}, we derive that we can find a positive $C_2$ such that
\begin{equation}\|{n}_\varepsilon(\cdot, t)\|_{L^{m+1}(\Omega)}\leq C_2~~ \mbox{for all}~~ t\in(0, T_{max,\varepsilon}).
 \label{ssfddf71cgghhvvkkhbnnhggffggjjllll}
\end{equation}
This combined with \dref{ssdffddfz2.571cghvvkbnnvvtthyybbhggllll} gives \dref{ssdffddfz2.571cghvvkbnnhggjjllll}  and finishes the proof of Lemma \ref{lemffggma456hjkhddff}.
\end{proof}

With Lemma \ref{lemffggma456hjkhddff} at hand, we can proceed to show that our approximate solutions are actually global
in time.
\begin{lemma}\label{kkklemmaghjmk4563025xxhjklojjkkk}
Let $m>\frac{10}{9}$. Then
for all $\varepsilon\in(0,1),$ the solution of  \dref{1.1fghyuisda} is global in time.
\end{lemma}
\begin{proof}
Assuming that $T_{max,\varepsilon}$ be finite for some $\varepsilon\in(0,1)$.
Firstly, testing the projected Stokes equation $u_{\varepsilon t}+A u_{\varepsilon}=\mathcal{P}(-\kappa (Y_{\varepsilon}u_{\varepsilon} \cdot \nabla)u_{\varepsilon}+n_{\varepsilon}\nabla \phi) $ by $Au_{\varepsilon}$ shows that
\begin{equation}
\begin{array}{rl}
&\disp{\frac{1}{{2}}\frac{d}{dt}\|A^{\frac{1}{2}}u_{\varepsilon}\|^{{{2}}}_{L^{{2}}(\Omega)}+
\int_{\Omega}|Au_{\varepsilon}|^2 }\\
=&\disp{ \int_{\Omega}Au_{\varepsilon}\kappa
(Y_{\varepsilon}u_{\varepsilon} \cdot \nabla)u_{\varepsilon}+ \int_{\Omega}n_{\varepsilon}\nabla\phi Au_{\varepsilon}}\\
\leq&\disp{ \frac{3}{4}\int_{\Omega}|Au_{\varepsilon}|^2+\kappa^2\int_{\Omega}
|(Y_{\varepsilon}u_{\varepsilon} \cdot \nabla)u_{\varepsilon}|^2+ \|\nabla\phi\|^2_{L^\infty(\Omega)}\int_{\Omega}n_{\varepsilon}^2
~~\mbox{for all}~~t\in(0,T_{max,\varepsilon}),}\\
\end{array}
\label{ddfghgghjjnnhhkklld911cz2.5ghju48}
\end{equation}
 where $\mathcal{P}$ denotes the Helmholtz project from $L^2(\Omega)$ into $L^2_\sigma(\Omega)$.
Next, we observe that
$D(1 + \varepsilon A)  :=W^{2,2}(\Omega) \cap W_{0,\sigma}^{1,2}(\Omega)\hookrightarrow L^\infty(\Omega),$ we can find $C_3 > 0$ and $C_4 > 0$ such that
\begin{equation}
\|Y_{\varepsilon}u_{\varepsilon}(\cdot,t)\|_{L^\infty(\Omega)}=\|(I+\varepsilon A)^{-1}u_{\varepsilon}(\cdot,t)\|_{L^\infty(\Omega)}\leq C_3\|u_{\varepsilon}(\cdot,t)\|_{L^2(\Omega)}\leq C_4~~\mbox{for all}~~t\in(0,T_{max,\varepsilon}).
\label{ssdcfvgdhhjjdfghgghjjnnhhkklld911cz2.5ghju48}
\end{equation}
Now, we derive from the H\"{o}lder inequality, \dref{czfvgb2.5ghhjuyuiihjj} (or \dref{czfvgb2.5ghhjuyuccvviihjj}) and \dref{ssdcfvgdhhjjdfghgghjjnnhhkklld911cz2.5ghju48} that
\begin{equation}
\begin{array}{rl}
&\|Y_{\varepsilon}(u_{\varepsilon}(\cdot,t) \cdot \nabla)u_{\varepsilon}(\cdot,t)\|_{L^{p}(\Omega)}\\
\leq&\|Y_{\varepsilon}u_{\varepsilon}(\cdot,t)\|_{L^{\infty}(\Omega)} \|u_{\varepsilon}(\cdot,t)\|_{L^{2}(\Omega)}|\Omega|^{\frac{2-p}{2}}\\
\leq & C_5~~\mbox{for all}~~t\in(0,T_{max,\varepsilon}).
\end{array}
\label{zjccfbhhccvbz2.5297x9630111kkuu}
\end{equation}

On the other hand, by 
\dref{ssdcfvgdhhjjdfghgghjjnnhhkklld911cz2.5ghju48}, we derive that
\begin{equation}
\begin{array}{rl}
\kappa^2\disp\int_{\Omega}
|(Y_{\varepsilon}u_{\varepsilon} \cdot \nabla)u_{\varepsilon}|^2\leq&\disp{ \kappa^2\|Y_{\varepsilon}u_{\varepsilon}\|^2_{L^\infty(\Omega)}\int_{\Omega}|\nabla u_{\varepsilon}|^2}\\
\leq&\disp{ \kappa^2\|Y_{\varepsilon}u_{\varepsilon}\|^2_{L^\infty(\Omega)}\int_{\Omega}|\nabla u_{\varepsilon}|^2}\\
\leq&\disp{ \kappa^2C_4^2\int_{\Omega}|\nabla u_{\varepsilon}|^2~~\mbox{for all}~~t\in(0,T_{max,\varepsilon}),}\\
\end{array}
\label{ssdcfvgddfghgghjjnnhhkklld911cz2.5ghju48}
\end{equation}
where $C_4$ is the same as \dref{ssdcfvgdhhjjdfghgghjjnnhhkklld911cz2.5ghju48}.

Plugging 
substitution of
\dref{ssdcfvgddfghgghjjnnhhkklld911cz2.5ghju48} into \dref{ddfghgghjjnnhhkklld911cz2.5ghju48}, we derive from  \dref{ssdffddfz2.571cghvvgjjllll} and the Gronwall inequality
   that there exists a positive constant $C_6$ such that
   \begin{equation}
\begin{array}{rl}
\disp\int_{\Omega}|\nabla u_{\varepsilon}(\cdot,t)|^2\leq C_6~~\mbox{for all}~~t\in(0,T_{max,\varepsilon}).
\end{array}
\label{ssdcfvgddfghgssdrrttghjjnnhhkklld911cz2.5ghju48}
\end{equation}
Let $h_{\varepsilon}(x,t)=\mathcal{P}[-\kappa (Y_{\varepsilon}u_{\varepsilon} \cdot \nabla)u_{\varepsilon}+n_{\varepsilon}\nabla \phi ]$.
Then along with \dref{ssdffddfz2.571cghvvgjjllll}  and \dref{ssdcfvgddfghgghjjnnhhkklld911cz2.5ghju48}--\dref{ssdcfvgddfghgssdrrttghjjnnhhkklld911cz2.5ghju48}, this in turn provides $C_7 > 0$ such that $\|h_{\varepsilon}(\cdot,t)\|_{L^2(\Omega)} \leq C_7$ for
all $t\in (0, T_{max,\varepsilon})$. Thus if we pick an arbitrary $\gamma\in (\frac{3}{4}, 1),$ then by smoothing properties of the
Stokes semigroup (\cite{Giga1215}) entail that for some $C_{8} > 0$, we have
\begin{equation}
\begin{array}{rl}
\|A^\gamma u_\varepsilon(\cdot, t)\|_{L^2(\Omega)}\leq&\disp{\|A^\gamma
e^{-tA}u_0\|_{L^2(\Omega)} +\int_0^t\|A^\gamma e^{-(t-\tau)A}h_\varepsilon(\cdot,\tau)d\tau\|_{L^2(\Omega)}d\tau}\\
\leq&\disp{C_2 t^{-\lambda_1(t-1)}
\|u_0\|_{L^2(\Omega)} +C_8\int_0^t(t-\tau)^{-\gamma}\|h_\varepsilon(\cdot,\tau)\|_{L^2(\Omega)}d\tau}\\
\leq&\disp{C_2 t^{-\lambda_1(t-1)}
\|u_0\|_{L^2(\Omega)} +\frac{C_7C_8T^{1-\gamma}_{max,\varepsilon}}{1-\gamma}~~ \mbox{for all}~~ t\in(0,T_{max,\varepsilon}).}\\
\end{array}
\label{cz2.57151ccvvhccvvhjjjkkhhggjjllll}
\end{equation}
Since $\gamma>\frac{3}{4},$
 $D(A^\gamma)$ is continuously embedded into $L^\infty(\Omega)$, hence, \dref{cz2.57151ccvvhccvvhjjjkkhhggjjllll} yields to
 \begin{equation}
\begin{array}{rl}
\|u_\varepsilon(\cdot, t)\|_{L^\infty(\Omega)}\leq C_9\|A^\gamma u_\varepsilon(\cdot, t)\|_{L^2(\Omega)}\leq C_{10}~~ \mbox{for all}~~ t\in(0,T_{max,\varepsilon})\\
\end{array}
\label{cz2.5715jkkcvccvvhjjjkddfffffkhhgll}
\end{equation}
for some positive constants $C_9$ and $C_{10}$.
Next, let  $T\in (0, T_{max,\varepsilon})$ and  $M(T) :=
\sup_{t\in(0,T )} \|\nabla c_\varepsilon(\cdot,t)\|_{L^4(\Omega)} $.
 Now, employing $\Delta$ to both sides of the
variation-of-constants formula for $c_{\varepsilon}$,
we derive that
 %
%
%
%
%
$$c_\varepsilon(\cdot, t) = e^{t\Delta }c_0 -\int_0^te^{-(t-s)\Delta}(n_\varepsilon c_\varepsilon+u_{\varepsilon} \cdot \nabla c_{\varepsilon})(\cdot,s)ds~~ \mbox{for all}~~ t\in(0,T_{max,\varepsilon}),$$
hence,
\begin{equation}
\begin{array}{rl}
&\disp{\|\nabla c_\varepsilon(\cdot, t)\|_{L^4(\Omega)}}\\
\leq&\disp{\|\nabla
e^{t\Delta}c_0\|_{L^{4}(\Omega)}}\\
&+\disp{\int_{0}^t\|\nabla e^{(t-s)\Delta}(n_\varepsilon c_\varepsilon)(\cdot,s)\|_{L^4(\Omega)}ds+\int_{0}^t\|\nabla e^{(t-s)\Delta}(u_\varepsilon \cdot \nabla c_\varepsilon)(\cdot,s)\|_{L^4(\Omega)}ds}\\
\end{array}
\label{dfff111zjccffgbhhjvccccvvvbbvscz2.5297x9630111kkuu}
\end{equation}
for all $t\in(0,T_{max,\varepsilon})$.
In the following, we will
estimate the right-hand side of \dref{dfff111zjccffgbhhjvccccvvvbbvscz2.5297x9630111kkuu}.

Indeed, 
due to the hypothesis of $c_0$ and the $L^p$-$L^q$ estimates we conclude  that there exists $C_{11}>0$ such that
\begin{equation}
\begin{array}{rl}
&\disp{\|\nabla
e^{t\Delta}c_0\|_{L^{4}(\Omega)}\leq C_{11}t^{-\frac{1}{2}}\|c_0\|_{L^{4}(\Omega)}~~\mbox{for all}~~ t>0.}\\
\end{array}
\label{zjccffgccvbbbccvvhhjvccvbbvscz2.5297x9630111kkuu}
\end{equation}
Since, $-\frac{1}{2}-\frac{3}{2}(\frac{1}{2}-\frac{1}{4})>-1,$  by $L^p$-$L^q$ estimate for Neumann semigroup and Lemma
 \ref{ghjssdeedrfe116lemma70hhjj} and Lemma \ref{lemffggma456hjkhddff},
we can  find $C_{12}> 0,C_{13}> 0$ and $\lambda_1> 0$ such that
\begin{equation}
\begin{array}{rl}
&\disp{\int_{0}^t\|\nabla e^{(t-s)\Delta}(n_\varepsilon c_\varepsilon)(\cdot,s)\|_{L^4(\Omega)}ds}\\
\leq&\disp{\int_{0}^tC_{12}(1+(t-s)^{-\frac{1}{2}-\frac{3}{2}(\frac{1}{2}-\frac{1}{4})})e^{-\lambda_1(t-s)}\|n_\varepsilon(\cdot,s) c_\varepsilon(\cdot,s)
\|_{L^2(\Omega)}ds}\\
\leq&\disp{\int_{0}^tC_{12}(1+(t-s)^{-\frac{1}{2}-\frac{3}{2}(\frac{1}{2}-\frac{1}{4})})e^{-\lambda_1(t-s)}\|n_\varepsilon(\cdot,s)\|_{L^2(\Omega)} \|c_\varepsilon(\cdot,s)
\|_{L^\infty(\Omega)}ds}\\
\leq&\disp{C_{13}~~ \mbox{for all}~~ t\in(0,T_{max,\varepsilon}).}\\
\end{array}
\label{zjccffgbhhjcccvvbvsc297x963011kuu}
\end{equation}

Now, with the help of the  H\"{o}lder inequality, we conclude that there exists a positive constant $C_{14}$ such that
\begin{equation}
\begin{array}{rl}
&\disp{\int_{0}^t\|\nabla e^{(t-s)\Delta}(u_\varepsilon \cdot \nabla c_\varepsilon)(\cdot,s)\|_{L^4(\Omega)}ds}\\
\leq&\disp{\int_{0}^tC_{14}(1+(t-s)^{-\frac{1}{2}-\frac{3}{2}(\frac{5}{18}-\frac{1}{4})})e^{-\lambda_1(t-s)}\|u_\varepsilon(\cdot,s)\nabla c_\varepsilon(\cdot,s)
\|_{L^{\frac{18}{5}}(\Omega)}ds
}\\
\leq&\disp{\int_{0}^tC_{14}(1+(t-s)^{-\frac{1}{2}-\frac{3}{2}(\frac{5}{18}-\frac{1}{4})})e^{-\lambda_1(t-s)}\|u_\varepsilon(\cdot,s)\|_{L^\infty(\Omega)}\|\nabla c_\varepsilon(\cdot,s)
\|_{L^{\frac{18}{5}}(\Omega)}ds}\\
\end{array}
\label{ddfzjccffgbhhjvcvvbbvz2.5297x9630111kkuu}
\end{equation}
for all $t\in(0,T_{max,\varepsilon})$.
On the other hand, due to the interpolation
inequality, we get that
$$
\begin{array}{rl}
&\disp{\|\nabla c_\varepsilon(\cdot,s)
\|_{L^{\frac{18}{5}}(\Omega)}}\\
\leq&\disp{C_{14}[
\|\nabla c_\varepsilon(\cdot,s)\|_{L^4(\Omega)}^{\frac{2}{3}}\| c_\varepsilon(\cdot,s)\|_{L^\infty(\Omega)}^{\frac{1}{3}}+\| c_\varepsilon(\cdot,s)\|_{L^\infty(\Omega)}]
~~ \mbox{for all}~~ t\in(0,T_{max,\varepsilon}).}\\
\end{array}
$$
Plugging  the above inequality  into \dref{ddfzjccffgbhhjvcvvbbvz2.5297x9630111kkuu} and applying \dref{cz2.5715jkkcvccvvhjjjkddfffffkhhgll}, we have
\begin{equation}
\begin{array}{rl}
&\disp{\int_{0}^t\|\nabla e^{(t-s)\Delta}(u_\varepsilon \cdot \nabla c_\varepsilon)(\cdot,s)\|_{L^4(\Omega)}ds\leq C_{15}M^{\frac{2}{3}}(T)+C_{15}~~ \mbox{for all}~~ t\in(0,T_{max,\varepsilon})
}\\
\end{array}
\label{zjccffgbhhjvvbbvz2.5297x96301uu}
\end{equation}
and some positive constant $C_{15}.$
Now, collecting \dref{dfff111zjccffgbhhjvccccvvvbbvscz2.5297x9630111kkuu}--\dref{ddfzjccffgbhhjvcvvbbvz2.5297x9630111kkuu} and \dref{zjccffgbhhjvvbbvz2.5297x96301uu}, we can derive
%
\begin{equation}
\begin{array}{rl}
&\disp{\|\nabla c_\varepsilon(\cdot, t)\|_{L^4(\Omega)}\leq C_{16}~~ \mbox{for all}~~ t\in(\tau,T_{max,\varepsilon})}\\
\end{array}
\label{zjccffgbhhjvccccvvvbbvscz2.5297x9630111kkuu}
\end{equation}
with $\tau\in(0,T_{max,\varepsilon})$ and some positive constant $C_{16}$.
In order to get the boundedness of $\|\nabla c_\varepsilon(\cdot, t)\|_{L^\infty(\Omega)}$,
we rewrite the variation-of-constants formula for $c_{\varepsilon}$ in the form
$$c_\varepsilon(\cdot, t) = e^{t(\Delta-1) }c_0 +\int_0^te^{(t-s)(\Delta-1)}(c_\varepsilon-n_\varepsilon c_\varepsilon-u_{\varepsilon} \cdot \nabla c_{\varepsilon})(\cdot,s)ds~~ \mbox{for all}~~ t\in(0,T_{max,\varepsilon}).$$
Now, picking $\theta\in(\frac{1}{2}+\frac{3}{2q_0},1),$ then the domain of the fractional power $D((-\Delta + 1)^\theta)\hookrightarrow W^{1,\infty}(\Omega)$
(\cite{Zhangddff4556}), where $q_0:=\min\{p_0,4\}>3$ and $p_0$ is the same as \dref{ssdffddfz2.571cghvvgjjllll}.

 Hence, in view of $L^p$-$L^q$ estimates associated heat semigroup, \dref{ccvvx1.731426677gg}, \dref{ssdffddfz2.571cghvvgjjllll}, \dref{cz2.5715jkkcvccvvhjjjkddfffffkhhgll} and \dref{zjccffgbhhjvccccvvvbbvscz2.5297x9630111kkuu}, we conclude that
\begin{equation}
\begin{array}{rl}
&\disp{\|\nabla c_\varepsilon(\cdot, t)\|_{W^{1,\infty}(\Omega)}}\\
\leq&\disp{C_{17}t^{-\theta}e^{-\lambda t}\|c_0\|_{L^{q_0}(\Omega)}}\\
&+\disp{\int_{0}^t(t-s)^{-\theta}e^{-\lambda(t-s)}\|(c_\varepsilon-n_\varepsilon c_\varepsilon-u_{\varepsilon} \cdot \nabla c_{\varepsilon})(s)\|_{L^{q_0}(\Omega)}ds}\\
\leq&\disp{C_{18}\tau^{-\theta}+C_{18}\int_{0}^t(t-s)^{-\theta}e^{-\lambda(t-s)}+C_{18}\int_{0}^t(t-s)^{-\theta}e^{-\lambda(t-s)}
[\|n_\varepsilon(s)\|_{L^{q_0}(\Omega)}+
\|\nabla c_{\varepsilon}(s)\|_{L^{q_0}(\Omega)}]ds}\\
\leq&\disp{C_{19}~~ \mbox{for all}~~ t\in(\tau,T_{max,\varepsilon})}\\
\end{array}
\label{zjccffgbhjcvvvbscz2.5297x96301ku}
\end{equation}
for some positive constant $C_{17},C_{18}$ and $C_{19}$.

Finally, for all $p>1$, multiplying the first equation in \dref{hhjjcz2.5} by $n_{\varepsilon}^{p-1}$, after integrating by parts and using the Young inequality,
 we easily  deduce from  \dref{zjccffgbhjcvvvbscz2.5297x96301ku} that
%
%
\begin{equation}
\begin{array}{rl}
\disp\frac{1}{{p}}\disp\frac{d}{dt}\|n_{\varepsilon}\|^{{{p}}}_{L^{{p}}(\Omega)}+
\disp\frac{C_D(p-1)}{2}\disp\int_{\Omega}n_{\varepsilon}^{m+p-3} |\nabla n_{\varepsilon}|^2
\leq &\disp{C_{20}\int_\Omega n_{\varepsilon}^{p+1-m}  }\\
\leq &\disp{C_{21}\int_\Omega n_{\varepsilon}^{p}+C_{14}~~ \mbox{for all}~~ t\in(\tau,T_{max,\varepsilon}) }\\
\end{array}
\label{cz2.5ghhjui78jj90099}
\end{equation}
and some positive constants $C_{20}$ and $C_{21}$.

Therefore, integrating the above inequality  with respect to $t$, we derive that there exists a positive constant $C_{22}$ such that
\begin{equation}
\begin{array}{rl}
\|n_{\varepsilon}(\cdot, t)\|_{L^{{p}}(\Omega)}\leq C_{22} ~~ \mbox{for all}~~p\geq1~~\mbox{and}~~  t\in(\tau,T_{max,\varepsilon}). \\
\end{array}
\label{cz2.5g556789hhjui78jj90099}
\end{equation}
Next, using the outcome of \dref{cz2.5g556789hhjui78jj90099} with suitably large $p$ as a starting point, we may invoke
Lemma A.1 in \cite{Tao794} which by means of a Moser-type iteration applied to the first equation in \dref{1.1fghyuisda}
establishes
\begin{equation}
\begin{array}{rl}
\|n_{\varepsilon}(\cdot, t)\|_{L^{{\infty}}(\Omega)}\leq C_{23} ~~ \mbox{for all}~~~  t\in(\tau,T_{max,\varepsilon}) \\
\end{array}
\label{cz2.5g5gghh56789hhjui78jj90099}
\end{equation}
and a positive constant $C_{23}$.
In view of \dref{cz2.5715jkkcvccvvhjjjkddfffffkhhgll},  \dref{zjccffgbhjcvvvbscz2.5297x96301ku} and \dref{cz2.5g5gghh56789hhjui78jj90099}, we apply Lemma \ref{lemma70} to reach a contradiction.
\end{proof}

\section{Time regularity}
In order to pass to the limit 
in \dref{1.1fghyuisda}, we shall need an appropriate boundedness
property of the time derivatives of certain powers of $n_{\varepsilon},c_{\varepsilon}$ and $u_{\varepsilon}$.
We first give the following lemma, which gives some
estimates for $n_{\varepsilon},c_{\varepsilon}$ and $n_{\varepsilon}$.
\begin{lemma}\label{lemma45630hhuujjuuyytt}
Let
\dref{dd1.1fghyuisdakkkllljjjkk} and \dref{ccvvx1.731426677gg}
 hold, and suppose that $m$ and $S$ satisfy \dref{ghnjmk9161gyyhuug} and \dref{x1.73142vghf48rtgyhu}--\dref{x1.73142vghf48gg}, respectively.
%
%
%
%
%
Then any small $\varepsilon>0(\varepsilon<1)$,
 one can find $C > 0$ independent of $\varepsilon$ such that for all $T\in (0, \infty)$
 \begin{equation}
\int_0^T\int_{\Omega}\left[|\nabla u_\varepsilon|^2+ |u_{\varepsilon}|^{\frac{10}{3}} \right]  \leq C(T+1).
\label{gghhzjscz2.ddffg5297x9630111kkhhiioott4}
\end{equation}

 Moreover, if $\frac{10}{9}< m \leq 2, $ then we have 
\begin{equation}
\int_0^T\int_{\Omega}\left[(n_\varepsilon+\varepsilon)^{\frac{3m+2}{3}} +|\nabla n_\varepsilon|^{\frac{3m+2}{4}}++|\nabla {c_{\varepsilon}}|^4+|D_{\varepsilon}(n_{\varepsilon})\nabla n_{\varepsilon}|^{\frac{3m+2}{3m+1}} \right] \leq C(T+1).
\label{gghhzjscz2.5297x9630111kkhhiioottvvbb}
\end{equation}
%
%
%
While if $m>2$, then
there exists $C>0$ independent of $\varepsilon$ such that
\begin{equation}
\begin{array}{rl}
&\disp{\int_{\Omega}\left[(n_{\varepsilon}+\varepsilon)^{m-1}+   c_{\varepsilon}^2+ | {u_{\varepsilon}}|^2\right]\leq C~~~\mbox{for all}~~ t>0}\\
\end{array}
\label{czfvgb2.jjkkkggffgghh5ghhjuyuccvviihjj}
\end{equation}
and
\begin{equation}
\begin{array}{rl}
&\disp{\int_{0}^T\int_{\Omega}\left[ (n_{\varepsilon}+\varepsilon)^{\frac{8(m-1)}{3}}+ (n_{\varepsilon}+\varepsilon)^{2m-4} |\nabla {n_{\varepsilon}}|^2+  |\nabla {c_{\varepsilon}}|^2+  |\nabla {u_{\varepsilon}}|^2+|D_{\varepsilon}(n_{\varepsilon})\nabla n_{\varepsilon}|^{\frac{8(m-1)}{4m-1}}\right]}\\
\leq&\disp{ C(T+1)~~\mbox{for all}~~ T > 0.}\\
\end{array}
\label{bnmbncz2.5ghhjuggyuuyuivvffggbnnihjj}
\end{equation}
\end{lemma}
\begin{proof}
Case $\frac{10}{9}< m\leq2$: Due to Lemma \ref{lemmakkllgg4563025xxhjklojjkkk},
%
there exists $C_1>0$ such that the solution of \dref{1.1fghyuisda} satisfies
\begin{equation}
\begin{array}{rl}
&\disp{\int_{\Omega}n_{\varepsilon}\ln n_{\varepsilon}+\int_{\Omega}|\nabla\sqrt{c_{\varepsilon}}|^2+\int_{\Omega}|u_{\varepsilon}|^2\leq C_1~~
\mbox{for all}~~t>0}\\
\end{array}
\label{czfv2.5gh33456uyuiihjj}
\end{equation}
and
\begin{equation}
\begin{array}{rl}
&\disp{\int_{0}^T\int_{\Omega}  \left(|\nabla {u_{\varepsilon}}|^2+\frac{D_{\varepsilon}(n_{\varepsilon})|\nabla n_{\varepsilon}|^2}{n_{\varepsilon}}+n_{\varepsilon}^{m-2} |\nabla {n_{\varepsilon}}|^2+|\nabla {c_{\varepsilon}}|^4+c_{\varepsilon}|D^2\ln c_{\varepsilon}|^2\right)\leq C_1(T+1)}\\
\end{array}
\label{bnmbncz2.htt678hyugghiihjj}
\end{equation}
for all $T > 0.$
Now,
applying  the Gagliardo-Nirenberg inequality, \dref{vgbhssddaqwswddaassffssff3.10deerfgghhjuuloollgghhhyhh} and $\varepsilon<1,$ we derive that there exist $C_i(i=2\ldots 6)$  such that
\begin{equation}
\begin{array}{rl}
\disp\int_{0}^T\disp\int_{\Omega} (n_{\varepsilon}+\varepsilon)^{\frac{3m+2}{3}} =&\disp{\int_{0}^T\| {(n_{\varepsilon}+\varepsilon)^{\frac{m}{2}}}\|^{{\frac{2(3m+2)}{3m}}}_{L^{\frac{2(3m+2)}{3m}}(\Omega)}}\\
\leq&\disp{C_2\int_{0}^T\left(\| \nabla{(n_{\varepsilon}+\varepsilon)^{\frac{m}{2}}}\|^{2}_{L^{2}(\Omega)}\|{(n_{\varepsilon}+\varepsilon)^{\frac{m}{2}}}\|^{{\frac{4}{3m}}}_{L^{\frac{2}{m}}(\Omega)}+
\|{(n_{\varepsilon}+\varepsilon)^{\frac{m}{2}}}\|^{{\frac{2(3m+2)}{3m}}}_{L^{\frac{2}{m}}(\Omega)}\right)}\\
\leq&\disp{C_3\int_{0}^T\left(\| \nabla{(n_{\varepsilon}+\varepsilon)^{\frac{m}{2}}}\|^{2}_{L^{2}(\Omega)}[\int_{\Omega}n_{\varepsilon}+|\Omega|]^{\frac{2}{3}}
+
[\int_{\Omega}n_{\varepsilon}+|\Omega|]^{{\frac{3m+2}{3}}}\right)}\\
\leq&\disp{C_4(T+1)~~\mbox{for all}~~ T > 0}\\
\end{array}
\label{ddffbnmbncz2ddfvgbhh.htt678hyuiihjj}
\end{equation}
and
\begin{equation}
\begin{array}{rl}
\disp\int_{0}^T\disp\int_{\Omega} |u_{\varepsilon}|^{\frac{10}{3}} =&\disp{\int_{0}^T\| {u_{\varepsilon}}\|^{{\frac{10}{3}}}_{L^{\frac{10}{3}}(\Omega)}}\\
\leq&\disp{C_5\int_{0}^T\left(\| \nabla{u_{\varepsilon}}\|^{2}_{L^{2}(\Omega)}\|{u_{\varepsilon}}\|^{{\frac{4}{3}}}_{L^{2}(\Omega)}+
\|{u_{\varepsilon}}\|^{{\frac{10}{3}}}_{L^{2}(\Omega)}\right)}\\
\leq&\disp{C_6(T+1)~~\mbox{for all}~~ T > 0.}\\
\end{array}
\label{bnmbncz2ddfvgffghhbhh.htt678hyuiihjj}
\end{equation}
Now, the estimates \dref{bnmbncz2.htt678hyugghiihjj}--\dref{ddffbnmbncz2ddfvgbhh.htt678hyuiihjj} together with the Young inequality ensures
%
%
\begin{equation}
\begin{array}{rl}
\disp\int_{0}^T\disp\int_{\Omega} |\nabla n_{\varepsilon}|^{\frac{3m+2}{4}} \leq&\disp{C_7\left(\int_{0}^T\int_{\Omega}n_{\varepsilon}^{m-2} |\nabla {n_{\varepsilon}}|^2+\int_{0}^T\disp\int_{\Omega} n_{\varepsilon}^{\frac{3m+2}{3}}\right) }\\
\leq&\disp{C_8(T+1)~~\mbox{for all}~~ T > 0}\\
\end{array}
\label{bnmbncz2ddfvgffgghhbhh.htt678hyuiihjj}
\end{equation}
and some positive constants $C_7$ and $C_8.$
Utilizing  \dref{ghnjmk9161gyyhuug}, \dref{bnmbncz2.htt678hyugghiihjj} and the H\"{o}lder inequality, it yields from \dref{ddffbnmbncz2ddfvgbhh.htt678hyuiihjj} that
we can find $C_9>0$ and $C_{10} > 0$ such that
\begin{equation}
\begin{array}{rl}
\disp\int_{0}^T\disp\int_{\Omega}|D_{\varepsilon}(n_{\varepsilon})\nabla n_{\varepsilon}|^{\frac{3m+2}{3m+1}} \leq&\disp{\left[\int_{0}^T\disp\int_{\Omega}\frac{D_{\varepsilon}(n_{\varepsilon})|\nabla n_{\varepsilon}|^2}{n_{\varepsilon}}\right]^{\frac{3m+2}{6m+2}}\left[\int_{0}^T\disp\int_{\Omega}[D_{\varepsilon}(n_{\varepsilon}) n_{\varepsilon}]^{\frac{3m+2}{3m}}\right]^{\frac{3m}{6m+2}}}\\
\leq&\disp{C_{9}\left[\int_{0}^T\disp\int_{\Omega}\frac{D_{\varepsilon}(n_{\varepsilon})|\nabla n_{\varepsilon}|^2}{n_{\varepsilon}}\right]^{\frac{3m+2}{6m+2}}\left[\int_{0}^T\disp\int_{\Omega} (n_{\varepsilon}+\varepsilon)^{\frac{3m+2}{3}}\right]^{\frac{3m}{6m+2}}}\\
\leq&\disp{C_{10}(T+1)~~\mbox{for all}~~ T > 0.}\\
\end{array}
\label{ddffbnmbncz2ddfvgbhh.htt678ghhjjjddfghhhyuiihjj}
\end{equation}

Case $m>2$: By virtue of  \dref{ghnjmk9161gyyhuug} and Lemma
\ref{lemmaghjssddgghhmk4563025xxhjklojjkkk}, it follows that
\begin{equation}
\begin{array}{rl}
&\disp{\int_{\Omega}(n_{\varepsilon}+\varepsilon)^{m-1}+\int_{\Omega}   c_{\varepsilon}^2+\int_{\Omega}  | {u_{\varepsilon}}|^2\leq C_{11}~~~\mbox{for all}~~ t>0}\\
\end{array}
\label{czfvgb2.jjkkkgghh5ghhjuyuccvviihjj}
\end{equation}
and
\begin{equation}
\begin{array}{rl}
&\disp{\int_{0}^T\int_{\Omega} \left[ (D_{\varepsilon}(n_{\varepsilon}))^{\frac{2m-4}{m-1}}|\nabla n_{\varepsilon}|^2+(n_{\varepsilon}+\varepsilon)^{2m-4} |\nabla {n_{\varepsilon}}|^2+ |\nabla {c_{\varepsilon}}|^2+ |\nabla {u_{\varepsilon}}|^2\right]\leq C_{11}(T+1)}\\
\end{array}
\label{bnmbncz2.5ghhjuggyuuyuivvbnnihjj}
\end{equation}
for all $T > 0$ and a positive constant $C_{11}>0$ independent of $\varepsilon$.

Now, since \dref{czfvgb2.jjkkkgghh5ghhjuyuccvviihjj} and \dref{bnmbncz2.5ghhjuggyuuyuivvbnnihjj}, employing  the Gagliardo-Nirenberg inequality and the
H\"{o}lder inequality, we conclude  that there exist positive
 constants $C_{i}(i=12\ldots 15)$ 
such that
\begin{equation}
\begin{array}{rl}
\disp\int_{0}^T\disp\int_{\Omega}(n_{\varepsilon}+\varepsilon)^{\frac{8(m-1)}{3}} =&\disp{\int_{0}^T\| {(n_{\varepsilon}+\varepsilon)^{m-1}}\|^{{\frac{8}{3}}}_{L^{\frac{8}{3}}(\Omega)}}\\
\leq&\disp{C_{12}\int_{0}^T\left(\| \nabla{(n_{\varepsilon}+\varepsilon)^{m-1}}\|^{2}_{L^{2}(\Omega)}\|{(n_{\varepsilon}+\varepsilon)^{m-1}}\|^{{\frac{2}{3}}}_{L^{1}(\Omega)}+
\|{(n_{\varepsilon}+\varepsilon)^{m-1}}\|^{{\frac{8}{3}}}_{L^{1}(\Omega)}\right)}\\
\leq&\disp{C_{13}(T+1)~~\mbox{for all}~~ T > 0.}\\
\end{array}
\label{ddffbnmbncz2ddfvgbhh.htt678ddfghhhyuiihjj}
\end{equation}
and
\begin{equation}
\begin{array}{rl}
\disp\int_{0}^T\disp\int_{\Omega}|D_{\varepsilon}(n_{\varepsilon})\nabla n_{\varepsilon}|^{\frac{8(m-1)}{4m-1}} \leq&\disp{\left[\int_{0}^T\disp\int_{\Omega}(D_{\varepsilon}(n_{\varepsilon}))^{\frac{2m-4}{m-1}}|\nabla n_{\varepsilon}|^2\right]^{\frac{4(m-1)}{4m-1}} \left[\int_{0}^T\disp\int_{\Omega}[D_{\varepsilon}(n_{\varepsilon}) ]^{\frac{8}{3}}\right]^{\frac{3}{4m-1}} }\\
\leq&\disp{C_{14}\left[\int_{0}^T\disp\int_{\Omega}(D_{\varepsilon}(n_{\varepsilon}))^{\frac{2m-4}{m-1}}|\nabla n_{\varepsilon}|^2\right]^{\frac{4(m-1)}{4m-1}} \left[\int_{0}^T\disp\int_{\Omega}[n_{\varepsilon}+\varepsilon ]^{\frac{8(m-1)}{3}}\right]^{\frac{3}{4m-1}}}\\
\leq&\disp{C_{15}(T+1)~~\mbox{for all}~~ T > 0.}\\
\end{array}
\label{ddffbnmbncz2ddfvgffgtyybhh.htt678ghhjjjddfghhhyuiihjj}
\end{equation}
Collecting \dref{czfv2.5gh33456uyuiihjj}--\dref{ddffbnmbncz2ddfvgffgtyybhh.htt678ghhjjjddfghhhyuiihjj}, we can obtain \dref{gghhzjscz2.ddffg5297x9630111kkhhiioott4}--\dref{bnmbncz2.5ghhjuggyuuyuivvffggbnnihjj}.
\end{proof}

As a last preparation for main results, we intend to supplement Lemmata \ref{lemmakkllgg4563025xxhjklojjkkk}--\ref{lemmaghjssddgghhmk4563025xxhjklojjkkk}
with bounds on time-derivatives.
\begin{lemma}\label{lemma45630hhuujjuuyytt}
Let
\dref{dd1.1fghyuisdakkkllljjjkk} and \dref{ccvvx1.731426677gg}
 hold, and suppose that $m$ and $S$ satisfy \dref{ghnjmk9161gyyhuug} and \dref{x1.73142vghf48rtgyhu}--\dref{x1.73142vghf48gg}, respectively.
 Then for any $T>0, $
  one can find $C > 0$ independent if $\varepsilon$ such that 
\begin{equation}
 \left\{\begin{array}{ll}
 \disp\int_0^T\|\partial_tn_\varepsilon(\cdot,t)\|_{(W^{2,q}(\Omega))^*}dt  \leq C(T+1),~~\mbox{if}~~\frac{10}{9}< m<\frac{26}{21},\\
 \disp\int_0^T\|\partial_tn_\varepsilon(\cdot,t)\|_{(W^{1,3m+2}(\Omega))^*}^{\frac{3m+2}{3m+1}}dt  \leq C(T+1),~~\mbox{if}~~\frac{26}{21}\leq m\leq2,\\
  \disp\int_0^T\|\partial_tn_\varepsilon^{m-1}(\cdot,t)\|_{(W^{2,q}(\Omega))^*}dt  \leq C(T+1),~~\mbox{if}~~m>2\\
   \end{array}\right.\label{1.1ddfgeddvbnmklllhyuisda}
\end{equation}
as well as
\begin{equation}
 \left\{\begin{array}{ll}
 \disp\int_0^T\|\partial_tc_\varepsilon(\cdot,t)\|_{(W^{1,\frac{3m+2}{3m-1}}(\Omega))^*}^{\frac{3m+2}{3}}dt  \leq C(T+1),~~\mbox{if}~~\frac{10}{9}< m<2,\\
  \disp\int_0^T\|\partial_tc_\varepsilon(\cdot,t)\|_{(W^{1,2}(\Omega))^*}^2dt  \leq C(T+1),~~\mbox{if}~~m>2\\
   \end{array}\right.\label{1.1dddfgbhnjmdfgeddvbnmklllhyussddisda}
\end{equation}
and
\begin{equation}
 \left\{\begin{array}{ll}
 \disp\int_0^T\|\partial_tu_\varepsilon(\cdot,t)\|_{(W^{1,\frac{3m+2}{3m-1}}(\Omega))^*}^{\frac{3m+2}{3}}dt  \leq C(T+1),~~\mbox{if}~~\frac{10}{9}< m<\frac{4}{3},\\
  \disp\int_0^T\|\partial_tu_\varepsilon(\cdot,t)\|_{(W^{1,2}(\Omega))^*}^2dt  \leq C(T+1),~~\mbox{if}~~m\geq\frac{4}{3},\\
   \end{array}\right.\label{1.1dddfgkkllbhddffgggnjmdfgeddvbnmklllhyussddisda}
\end{equation}
where $q>3.$
\end{lemma}
\begin{proof}
Case $\frac{26}{21}\leq m\leq2$:
Firstly,  testing the first equation of \dref{1.1fghyuisda}
 by certain   $\varphi\in C^{\infty}(\bar{\Omega})$, we have
 \begin{equation}
\begin{array}{rl}
&\disp\left|\int_{\Omega}\partial_{t}n_{\varepsilon}(\cdot,t)\varphi\right|\\
 =&\disp{\left|\int_{\Omega}\left[\nabla\cdot(D_{\varepsilon}(n_{\varepsilon})\nabla n_{\varepsilon})-\nabla\cdot(n_{\varepsilon}F_{\varepsilon}(n_{\varepsilon})S_\varepsilon(x, n_{\varepsilon}, c_{\varepsilon})\cdot\nabla c_{\varepsilon})-u_{\varepsilon}\cdot\nabla n_{\varepsilon}\right]\cdot\varphi\right|}
\\
=&\disp{\left|-\int_\Omega D_{\varepsilon}(n_{\varepsilon})\nabla n_{\varepsilon}\cdot\nabla\varphi+\int_\Omega n_{\varepsilon}F_{\varepsilon}(n_{\varepsilon})S_\varepsilon(x, n_{\varepsilon}, c_{\varepsilon})\cdot\nabla c_{\varepsilon}\cdot\nabla\varphi+\int_\Omega n_{\varepsilon}u_\varepsilon\cdot\nabla\varphi\right|}\\
\leq&\disp{\left\{\|D_{\varepsilon}(n_{\varepsilon})\nabla n_{\varepsilon}\|_{L^{\frac{3m+2}{3m+1}}(\Omega)}+\|n_{\varepsilon}F_{\varepsilon}(n_{\varepsilon})S_\varepsilon(x, n_{\varepsilon}, c_{\varepsilon})\cdot\nabla c_{\varepsilon}\|_{L^{\frac{3m+2}{3m+1}}(\Omega)}+\|n_{\varepsilon}u_\varepsilon\|_{L^{\frac{3m+2}{3m+1}}(\Omega)}\right\}\|\varphi\|_{W^{1,3m+2}(\Omega)}
}\\
\end{array}
\label{gbhncvbmdcfvgcz2.5ghju48}
\end{equation}
for all $t>0$.
Hence, with the help of \dref{gghhzjscz2.5297x9630111kkhhiioottvvbb} and \dref{x1.73142vghf48gg}, we derive that
\begin{equation}
\begin{array}{rl}
&\disp\int_0^T\|\partial_{t}n_{\varepsilon}(\cdot,t)\|^{\frac{3m+2}{3m+1}}_{(W^{1,3m+2}(\Omega))^*}dt\\
\leq&\disp{\int_0^T\left\{\|D_{\varepsilon}(n_{\varepsilon})\nabla n_{\varepsilon}\|_{L^{\frac{3m+2}{3m+1}}(\Omega)}+\|n_{\varepsilon}F_{\varepsilon}(n_{\varepsilon})S_\varepsilon(x, n_{\varepsilon}, c_{\varepsilon})\cdot\nabla c_{\varepsilon}\|_{L^{\frac{3m+2}{3m+1}}(\Omega)}+\|n_{\varepsilon}u_\varepsilon\|_{L^{\frac{3m+2}{3m+1}}(\Omega)}\right\}^{\frac{3m+2}{3m+1}}dt
}\\
\leq&\disp{C_1\left\{\int_0^T\int_{\Omega}|D_{\varepsilon}(n_{\varepsilon})\nabla n_{\varepsilon}|^{\frac{3m+2}{3m+1}}+\int_0^T\int_{\Omega}|n_{\varepsilon}F_{\varepsilon}(n_{\varepsilon})S_\varepsilon(x, n_{\varepsilon}, c_{\varepsilon})\cdot\nabla c_{\varepsilon}|^{\frac{3m+2}{3m+1}}+\int_0^T\int_{\Omega}|n_{\varepsilon}u_\varepsilon|^{\frac{3m+2}{3m+1}}\right\}
}\\
\leq&\disp{C_2(T+1)+C_1[S(\|c_0\|_{L^\infty(\Omega)})]^{\frac{3m+2}{3m+1}}\int_0^T\int_{\Omega}|n_{\varepsilon} \nabla c_{\varepsilon}|^{\frac{3m+2}{3m+1}}dt+C_1\int_0^T\int_{\Omega}|n_{\varepsilon}u_\varepsilon|^{\frac{3m+2}{3m+1}}
}\\
\end{array}
\label{gbhncvbmdcxxcdfvgbfvgcz2.5ghju48}
\end{equation}
for all $T>0$ and some positive constants $C_1$ and $C_2$.
In what follows, we shall estimate each term on the right-hand side of \dref{gbhncvbmdcxxcdfvgbfvgcz2.5ghju48} term by term.
Next, applying \dref{gghhzjscz2.ddffg5297x9630111kkhhiioott4}, \dref{gghhzjscz2.5297x9630111kkhhiioottvvbb}, \dref{bnmbncz2.htt678hyugghiihjj}, the Young inequality and employing $m\geq\frac{26}{21}$, we conclude that
\begin{equation}
\begin{array}{rl}
\disp\int_0^T\disp\int_{\Omega}|n_{\varepsilon} \nabla c_{\varepsilon}|^{\frac{3m+2}{3m+1}}&\leq\disp{\int_0^T\int_{\Omega}n_{\varepsilon}^{\frac{3m+2}{3}}+ \int_0^T\int_{\Omega}|\nabla c_{\varepsilon}|^{\frac{3m+2}{3m-2}}
}\\
&\leq\disp{\int_0^T\int_{\Omega}n_{\varepsilon}^{\frac{3m+2}{3}}+ \int_0^T\int_{\Omega}|\nabla c_{\varepsilon}|^{4}+|\Omega|T
}\\
&\leq\disp{C_3(T+1)~~\mbox{for all}~~ T>0
}\\
\end{array}
\label{gbhncvbmdcxxcccvffhhjdfvgbfvgcz2.5ghju48}
\end{equation}
and
\begin{equation}
\begin{array}{rl}
\disp\int_0^T\disp\int_{\Omega}|n_{\varepsilon}u_\varepsilon|^{\frac{3m+2}{3m+1}}&\leq\disp{\int_0^T\int_{\Omega}n_{\varepsilon}^{\frac{3m+2}{3}}+ \int_0^T\int_{\Omega}|u_{\varepsilon}|^{\frac{3m+2}{3m-2}}
}\\
&\leq\disp{\int_0^T\int_{\Omega}n_{\varepsilon}^{\frac{3m+2}{3}}+ \int_0^T\int_{\Omega}|u_{\varepsilon}|^{\frac{10}{3}}+|\Omega|T
}\\
&\leq\disp{C_4(T+1)~~\mbox{for all}~~ T>0.
}\\
\end{array}
\label{gbhncvbmdcddddddddffffffffffxxcdfvgbfvgcz2.5ghju48}
\end{equation}
and some positive constants $C_3$ and $C_4$.
Now, collecting  \dref{gbhncvbmdcxxcdfvgbfvgcz2.5ghju48}--\dref{gbhncvbmdcddddddddffffffffffxxcdfvgbfvgcz2.5ghju48} yields to
\begin{equation}
\begin{array}{rl}
\disp\int_0^T\|\partial_{t}n_{\varepsilon}(\cdot,t)\|^{\frac{3m+2}{3m+1}}_{(W^{1,3m+2}(\Omega))^*}dt\leq
C_5(T+1)~~\mbox{for all}~~ T>0.
\\
\end{array}
\label{gbhncvbxxaasqqwmdcxxcdfvgbfvgcz2.5ghju48}
\end{equation}
and a constant $C_5>0.$

Case $\frac{10}{9}< m<\frac{26}{21}$: Similarly, we also derive that
\begin{equation}
\begin{array}{rl}
&\disp\left|\int_{\Omega}\partial_{t}n_{\varepsilon}(\cdot,t)\varphi\right|\\
\leq&\disp{\left\{\|D_{\varepsilon}(n_{\varepsilon})\nabla n_{\varepsilon}\|_{L^{1}(\Omega)}+\|n_{\varepsilon}F_{\varepsilon}(n_{\varepsilon})S_\varepsilon(x, n_{\varepsilon}, c_{\varepsilon})\cdot\nabla c_{\varepsilon}\|_{L^{1}(\Omega)}+\|n_{\varepsilon}u_\varepsilon\|_{L^{1}(\Omega)}\right\}\|\varphi\|_{W^{1,\infty}(\Omega)}
}\\
\end{array}
\label{yyyygbhncvbmdcfvgcz2.5ghju48}
\end{equation}
for all $t>0$.
Hence, due to the embedding $W^{2,q }(\Omega)\hookrightarrow  W^{1,\infty}(\Omega)(q>3)$, we deduce from $C_6$ and $C_7 $ such
that
\begin{equation}
\begin{array}{rl}
&\disp\int_0^T\|\partial_{t}n_{\varepsilon}(\cdot,t)\|_{(W^{2,q }(\Omega))^*}dt\\
\leq&\disp{C_6\left\{\int_0^T\int_{\Omega}|D_{\varepsilon}(n_{\varepsilon})\nabla n_{\varepsilon}|^{\frac{3m+2}{3m+1}}+\int_0^T\int_{\Omega}|\nabla c_{\varepsilon}|^{4}+\int_0^T\int_{\Omega}n_{\varepsilon}^{\frac{3m+2}{3}}+\int_0^T\int_{\Omega}|u_\varepsilon|^{\frac{10}{3}}+T+1\right\}
}\\
\leq&\disp{C_7(T+1)
}\\
\end{array}
\label{yyygbhncvbmdcxxcdfvgbfvgcz2.5ghju48}
\end{equation}
for all $T>0$. 
Next, in view of \dref{gghhzjscz2.ddffg5297x9630111kkhhiioott4}, \dref{gghhzjscz2.5297x9630111kkhhiioottvvbb}, \dref{bnmbncz2.htt678hyugghiihjj}, the Young inequality and  $m\geq\frac{26}{21}$, we may find some positive constants $C_8$ and $C_9$ such that
\begin{equation}
\begin{array}{rl}
\disp\int_0^T\disp\int_{\Omega}|n_{\varepsilon} \nabla c_{\varepsilon}|^{\frac{3m+2}{3m+1}}&\leq\disp{\int_0^T\int_{\Omega}n_{\varepsilon}^{\frac{3m+2}{3}}+ \int_0^T\int_{\Omega}|\nabla c_{\varepsilon}|^{\frac{3m+2}{3m-2}}
}\\
&\leq\disp{\int_0^T\int_{\Omega}n_{\varepsilon}^{\frac{3m+2}{3}}+ \int_0^T\int_{\Omega}|\nabla c_{\varepsilon}|^{4}+|\Omega|T
}\\
&\leq\disp{C_8(T+1)~~\mbox{for all}~~ T>0
}\\
\end{array}
\label{yyyygbhncvbmdcxxcccvffhhjdfvgbfvgcz2.5ghju48}
\end{equation}
and
\begin{equation}
\begin{array}{rl}
\disp\int_0^T\disp\int_{\Omega}|n_{\varepsilon}u_\varepsilon|^{\frac{3m+2}{3m+1}}&\leq\disp{\int_0^T\int_{\Omega}n_{\varepsilon}^{\frac{3m+2}{3}}+ \int_0^T\int_{\Omega}|u_{\varepsilon}|^{\frac{3m+2}{3m-2}}
}\\
&\leq\disp{\int_0^T\int_{\Omega}n_{\varepsilon}^{\frac{3m+2}{3}}+ \int_0^T\int_{\Omega}|u_{\varepsilon}|^{\frac{10}{3}}+|\Omega|T
}\\
&\leq\disp{C_9(T+1)~~\mbox{for all}~~ T>0.
}\\
\end{array}
\label{yyyygbhncvbmdcddddddddffffffffffxxcdfvgbfvgcz2.5ghju48}
\end{equation}

Case $m>2$:
Next,
testing the first equation of \dref{1.1fghyuisda}
 by certain   $({m-1})n_{\varepsilon}^{m-2}\varphi\in C^{\infty}(\bar{\Omega})$ and using \dref{ghnjmk9161gyyhuug}, we have
 \begin{equation}
\begin{array}{rl}
&\disp\left|\int_{\Omega}(n_{\varepsilon}^{m-1})_{t}\varphi\right|\\
 =&\disp{\left|\int_{\Omega}\left[\nabla\cdot(D_{\varepsilon}(n_{\varepsilon})\nabla n_{\varepsilon})-\nabla\cdot(n_{\varepsilon}F_{\varepsilon}(n_{\varepsilon})S_\varepsilon(x, n_{\varepsilon}, c_{\varepsilon})\cdot\nabla c_{\varepsilon})-u_{\varepsilon}\cdot\nabla n_{\varepsilon}\right]\cdot({m-1})n_{\varepsilon}^{m-2}\varphi\right|}
\\
\leq&\disp{\left|-(m-1)\int_\Omega \left[D_{\varepsilon}(n_{\varepsilon})n_{\varepsilon}^{m-2}\nabla n_{\varepsilon}\cdot\nabla\varphi+(m-2) D_{\varepsilon}(n_{\varepsilon})n_{\varepsilon}^{m-3}|\nabla n_{\varepsilon}|^2\varphi\right]\right|}\\
&+\disp{(m-1)\left|\int_\Omega[(m-2) n_{\varepsilon}^{m-2}F_{\varepsilon}(n_{\varepsilon})S_\varepsilon(x, n_{\varepsilon}, c_{\varepsilon})\nabla n_{\varepsilon}\cdot\nabla c_{\varepsilon}\varphi+ n_{\varepsilon}^{m-1}F_{\varepsilon}(n_{\varepsilon})S_\varepsilon(x, n_{\varepsilon}, c_{\varepsilon})\nabla c_{\varepsilon}\cdot\nabla \varphi]\right|}\\
&\disp{+\left|\int_\Omega n_{\varepsilon}^{m-1}u_\varepsilon\cdot\nabla\varphi\right|}\\
\leq&\disp{(m-1)\int_\Omega \left[C_{\bar{D}}(n_{\varepsilon}+\varepsilon)^{m-1}n_{\varepsilon}^{m-2}|\nabla n_{\varepsilon}||\nabla\varphi|+(m-2)C_{\bar{D}} (n_{\varepsilon}+\varepsilon)^{m-1}n_{\varepsilon}^{m-3}|\nabla n_{\varepsilon}|^2|\varphi|\right]}\\
&+\disp{(m-1)\left|\int_\Omega[(m-2) n_{\varepsilon}^{m-2}F_{\varepsilon}(n_{\varepsilon})S_\varepsilon(x, n_{\varepsilon}, c_{\varepsilon})\nabla n_{\varepsilon}\cdot\nabla c_{\varepsilon}\varphi+ n_{\varepsilon}^{m-1}F_{\varepsilon}(n_{\varepsilon})S_\varepsilon(x, n_{\varepsilon}, c_{\varepsilon})\nabla c_{\varepsilon}\cdot\nabla \varphi]\right|}\\
&\disp{+\left|\int_\Omega n_{\varepsilon}^{m-1}u_\varepsilon\cdot\nabla\varphi\right|}\\
\leq&\disp{m(m-1)C_{\bar{D}}\left\{\int_\Omega \left[(n_{\varepsilon}+\varepsilon)^{m-1}n_{\varepsilon}^{m-2}|\nabla n_{\varepsilon}|+ (n_{\varepsilon}+\varepsilon)^{m-1}n_{\varepsilon}^{m-3}|\nabla n_{\varepsilon}|^2\right]\right\}\|\varphi\|_{W^{1,\infty}(\Omega)}}\\
&+\disp{(m-1)^2[S(\|c_0\|_{L^\infty(\Omega)})+1]\left\{\int_\Omega[ n_{\varepsilon}^{m-2}|\nabla n_{\varepsilon}||\nabla c_{\varepsilon}|+ n_{\varepsilon}^{m-1}|\nabla c_{\varepsilon}|+ n_{\varepsilon}^{m-1}|u_\varepsilon|]\right\}\|\varphi\|_{W^{1,\infty}(\Omega)}}\\
\end{array}
\label{gbhncvbmdcfvgcz2.5ghju48}
\end{equation}
for all $t>0$.
Hence, observe that the embedding $W^{2,q }(\Omega)\hookrightarrow  W^{1,\infty}(\Omega)(q>3)$, due to \dref{bnmbncz2.5ghhjuggyuuyuivvffggbnnihjj}, \dref{bnmbncz2.5ghhjuyuivvbnnihjj} and \dref{bnmbncz2ddfvgffghhbhh.htt678hyuiihjj}, applying $m>2$ and the Young inequality, we derive  $C_1,C_2$ and $C_3$ such
that
\begin{equation}
\begin{array}{rl}
&\disp\int_0^T\|\partial_{t}n_{\varepsilon}^{m-1}(\cdot,t)\|_{(W^{2,q }(\Omega))^*}dt\\
\leq&\disp{C_1\left\{\int_0^T\int_{\Omega}(n_{\varepsilon}+\varepsilon)^{2m-4}|\nabla n_{\varepsilon}|^{2}+\int_0^T\int_{\Omega}n_{\varepsilon}^{2m-2}+\int_0^T\int_{\Omega}|\nabla c_{\varepsilon}|^{2}+\int_0^T\int_{\Omega}|u_\varepsilon|^{2}\right\}
}\\
\leq&\disp{C_2\left\{\int_0^T\int_{\Omega}(n_{\varepsilon}+\varepsilon)^{2m-4}|\nabla n_{\varepsilon}|^{2}+\int_0^T\int_{\Omega}|\nabla c_{\varepsilon}|^{2}+\int_0^T\int_{\Omega}n_{\varepsilon}^{\frac{8(m-1)}{3}}+\int_0^T\int_{\Omega}|u_\varepsilon|^{\frac{10}{3}}+T\right\}
}\\
\leq&\disp{C_3(T+1)~~\mbox{for all}~~ T > 0,
}\\
\end{array}
\label{yyygbhncvbddffmdcxxcdfvgbddffvgcz2.5ghju48}
\end{equation}
which leads directly to
 \begin{equation}
 \begin{array}{ll}
  \disp\int_0^T\|\partial_tn_\varepsilon^{m-1}(\cdot,t)\|_{(W^{2,q}(\Omega))^*}dt  \leq C_4(T+1)\\
   \end{array}\label{1.1dhhttyuuiidfgeddvbnmklllhyuisda}
\end{equation}
for a positive constant $C_4.$

Now, in view of
\dref{gbhncvbxxaasqqwmdcxxcdfvgbfvgcz2.5ghju48},
\dref{yyygbhncvbmdcxxcdfvgbfvgcz2.5ghju48} and
\dref{1.1dhhttyuuiidfgeddvbnmklllhyuisda}, we can derive \dref{1.1ddfgeddvbnmklllhyuisda}.

Case $m>2:$
Likewise, given any $\varphi\in  C^\infty(\bar{\Omega})$, we may test the second equation in \dref{1.1fghyuisda} against
$\varphi$
to conclude  that
\begin{equation}
\begin{array}{rl}
\disp\left|\int_{\Omega}\partial_{t}c_{\varepsilon}(\cdot,t)\varphi\right|=&\disp{\left|\int_{\Omega}\left[\Delta c_{\varepsilon}-n_{\varepsilon}c_{\varepsilon}-u_{\varepsilon}\cdot\nabla c_{\varepsilon}\right]\cdot\varphi\right|}
\\
=&\disp{\left|-\int_\Omega \nabla c_{\varepsilon}\cdot\nabla\varphi-\int_\Omega n_{\varepsilon} c_{\varepsilon}\varphi+\int_\Omega c_{\varepsilon}u_\varepsilon\cdot\nabla\varphi\right|}\\
\leq&\disp{\left\{\|\nabla c_{\varepsilon}\|_{L^{2}(\Omega)}+\|n_{\varepsilon}c_{\varepsilon} \|_{L^{2}(\Omega)}+\|c_{\varepsilon}u_\varepsilon\|_{L^{2}(\Omega)}\right\}\|\varphi\|_{W^{1,2}(\Omega)}
~~\mbox{for all}~~ t>0,}\\
\end{array}
\label{gbhncvbmdcfxxdcvbccvbbvgcz2.5ghju48}
\end{equation}
from which, after using \dref{hnjmssddaqwswddaassffssff3.10deerfgghhjuuloollgghhhyhh}, \dref{gghhzjscz2.ddffg5297x9630111kkhhiioott4} and \dref{bnmbncz2.5ghhjuggyuuyuivvffggbnnihjj}, we conclude  that there exist positive constants $C_{15},C_{16}$ and $C_{17}$ such that
\begin{equation}
\begin{array}{rl}
&\disp\int_0^T\|\partial_{t}c_{\varepsilon}(\cdot,t)\|^{2}_{(W^{1,2}(\Omega))^*}dt\\
\leq&\disp{C_{15}\int_0^T\int_\Omega|\nabla c_{\varepsilon}|^{2}+C_{15}\int_0^T\int_\Omega n_{\varepsilon}^2+C_{15}\int_0^T\int_\Omega |u_\varepsilon|^2}\\
\leq&\disp{C_{15}\int_0^T\int_\Omega|\nabla c_{\varepsilon}|^{2}+C_{15}\|_{L^\infty(\Omega)}^2\int_0^T\int_\Omega n_{\varepsilon}^{\frac{8(m-1)}{3}}+C_{15}\|c_{0}\|_{L^\infty(\Omega)}^2\int_0^T\int_\Omega |u_\varepsilon|^{\frac{10}{3}}+C_{16}T}\\
\leq&\disp{C_{17}(T+1)
~~\mbox{for all}~~ T>0.}\\
\end{array}
\label{gbhncvbmdcfxxxcvbddfghnxdcvbbbvgcz2.5ghju48}
\end{equation}

Case $\frac{10}{9}<m\leq 2:$
For any given $t>0$ and $\varphi\in  C^\infty(\bar{\Omega})$, multiplying  the second equation of \dref{1.1fghyuisda} by
$\varphi$, we derive that
$$
\begin{array}{rl}
\disp\left|\int_{\Omega}\partial_{t}c_{\varepsilon}(\cdot,t)\varphi\right|\leq&\disp{\left\{\|\nabla c_{\varepsilon}\|_{L^{\frac{3m+2}{3}}(\Omega)}+
\|n_{\varepsilon}c_{\varepsilon} \|_{L^{\frac{3m+2}{3}}(\Omega)}+\|c_{\varepsilon}u_\varepsilon\|_{L^{\frac{3m+2}{3}}(\Omega)}\right\}
\|\varphi\|_{W^{1,\frac{3m+2}{3m-1}}(\Omega)}.}\\
\end{array}
$$
Hence, \dref{hnjmssddaqwswddaassffssff3.10deerfgghhjuuloollgghhhyhh}, \dref{gghhzjscz2.ddffg5297x9630111kkhhiioott4} and \dref{gghhzjscz2.5297x9630111kkhhiioottvvbb} imply   that
\begin{equation}
\begin{array}{rl}
&\disp\int_0^T\|\partial_{t}c_{\varepsilon}(\cdot,t)\|^{\frac{3m+2}{3}}_{(W^{1,\frac{3m+2}{3m-1}}(\Omega))^*}dt\\
\leq&\disp{C_{18}\int_0^T\int_\Omega|\nabla c_{\varepsilon}|^{\frac{3m+2}{3}}+C_{18}\int_0^T\int_\Omega n_{\varepsilon}^{\frac{3m+2}{3}}+C_{18}\int_0^T\int_\Omega |u_\varepsilon|^{\frac{3m+2}{3}}}\\
\leq&\disp{C_{18}\int_0^T\int_\Omega|\nabla c_{\varepsilon}|^{4}+C_{18}\int_0^T\int_\Omega n_{\varepsilon}^{\frac{3m+2}{3}}+C_{18}\int_0^T\int_\Omega |u_\varepsilon|^{\frac{10}{3}}+C_{19}T}\\
\leq&\disp{C_{20}(T+1)
~~\mbox{for all}~~ T>0.}\\
\end{array}
\label{gbhncvbmdcfxxdffxcvbnxdcvbbbvgcz2.5ghjddvccvvvbu48}
\end{equation}
and some positive constants $C_{18},C_{19}$ and $C_{20}.$

Now, combining   \dref{gbhncvbmdcfxxxcvbddfghnxdcvbbbvgcz2.5ghju48}  and \dref{gbhncvbmdcfxxdffxcvbnxdcvbbbvgcz2.5ghjddvccvvvbu48}, we can get \dref{1.1dddfgbhnjmdfgeddvbnmklllhyussddisda}.

Finally, for any given  $\varphi\in C^{\infty}_{0,\sigma} (\Omega;\mathbb{R}^3)$, we infer from the third equation in \dref{1.1fghyuisda} that
\begin{equation}
\begin{array}{rl}
\disp\left|\int_{\Omega}\partial_{t}u_{\varepsilon}(\cdot,t)\varphi\right|=&\disp{\left|-\int_\Omega \nabla u_{\varepsilon}\cdot\nabla\varphi-\int_\Omega (Y_{\varepsilon}u_{\varepsilon}\otimes u_{\varepsilon})\cdot\nabla\varphi+\int_\Omega n_{\varepsilon}\nabla \phi\cdot\varphi\right|
~~\mbox{for all}~~ t>0.}\\
\end{array}
\label{gbhncvbmdcfxxdcvbccvqqwerrbbvgcz2.5ghju48}
\end{equation}
Case $\frac{10}{9}<m<\frac{4}{3}$: Due to \dref{gghhzjscz2.ddffg5297x9630111kkhhiioott4}, \dref{gghhzjscz2.5297x9630111kkhhiioottvvbb} and \dref{ssdcfvgdhhjjdfghgghjjnnhhkklld911cz2.5ghju48}, there exist some positive constants $C_{21},C_{22}, C_{23}$ and $C_{24}$ such that
\begin{equation}
\begin{array}{rl}
&\disp\int_0^T\|\partial_{t}u_{\varepsilon}(\cdot,t)\|^{\frac{3m+2}{3}}_{(W^{1,\frac{3m+2}{3m-1}}(\Omega))^*}dt\\
\leq&\disp{C_{21}\int_0^T\int_\Omega|\nabla u_{\varepsilon}|^{\frac{3m+2}{3}}+C_{21}\int_0^T\int_\Omega |Y_{\varepsilon}u_{\varepsilon}\otimes u_{\varepsilon}|^{\frac{3m+2}{3}}+C_{21}\int_0^T\int_\Omega n_\varepsilon^{\frac{3m+2}{3}}}\\
\leq&\disp{C_{21}\int_0^T\int_\Omega|\nabla u_{\varepsilon}|^{2}+C_{21}\int_0^T\int_\Omega |Y_{\varepsilon}u_\varepsilon|^{2}+C_{21}\int_0^T\int_\Omega n_{\varepsilon}^{\frac{3m+2}{3}}+C_{22}T}\\
\leq&\disp{C_{21}\int_0^T\int_\Omega|\nabla u_{\varepsilon}|^{2}+C_{21}\int_0^T\int_\Omega |u_\varepsilon|^{\frac{10}{3}}+C_{21}\int_0^T\int_\Omega n_{\varepsilon}^{\frac{3m+2}{3}}+C_{23}T}\\
\leq&\disp{C_{24}(T+1)
~~\mbox{for all}~~ T>0.}\\
\end{array}
\label{gbhncvbmdcfxxdffxcvbnxdcyyuiivbbbvgcz2.5ghjddvccvvvbu48}
\end{equation}
Case $\frac{4}{3}\leq m\leq2$: Similarly,   by \dref{gghhzjscz2.ddffg5297x9630111kkhhiioott4}, \dref{gghhzjscz2.5297x9630111kkhhiioottvvbb} and \dref{ssdcfvgdhhjjdfghgghjjnnhhkklld911cz2.5ghju48}, we may find some positive constants $C_{25}, C_{26},C_{27}$ and $C_{28}$ such that
\begin{equation}
\begin{array}{rl}
&\disp\int_0^T\|\partial_{t}u_{\varepsilon}(\cdot,t)\|^{2}_{(W^{1,2}(\Omega))^*}dt\\
\leq&\disp{C_{25}\int_0^T\int_\Omega|\nabla u_{\varepsilon}|^{2}+C_{25}\int_0^T\int_\Omega |Y_{\varepsilon}u_{\varepsilon}\otimes u_{\varepsilon}|^{2}+C_{25}\int_0^T\int_\Omega n_\varepsilon^{2}}\\
\leq&\disp{C_{25}\int_0^T\int_\Omega|\nabla u_{\varepsilon}|^{2}+C_{25}\int_0^T\int_\Omega |Y_{\varepsilon}u_\varepsilon|^{2}+C_{25}\int_0^T\int_\Omega n_{\varepsilon}^{\frac{3m+2}{3}}+C_{26}T}\\
\leq&\disp{C_{25}\int_0^T\int_\Omega|\nabla u_{\varepsilon}|^{2}+C_{25}\int_0^T\int_\Omega |u_\varepsilon|^{\frac{10}{3}}+C_{25}\int_0^T\int_\Omega n_{\varepsilon}^{\frac{3m+2}{3}}+C_{27}T}\\
\leq&\disp{C_{28}(T+1)
~~\mbox{for all}~~ T>0.}\\
\end{array}
\label{gbhncvbmdcfxxxcvbddfghnxdckkvbgtyyiiobddfffbvgcz2.5ghju48}
\end{equation}

Case $m>2$: Now,   in view of \dref{gghhzjscz2.ddffg5297x9630111kkhhiioott4}, \dref{bnmbncz2.5ghhjuggyuuyuivvffggbnnihjj} and \dref{ssdcfvgdhhjjdfghgghjjnnhhkklld911cz2.5ghju48}, we also derive that
\begin{equation}
\begin{array}{rl}
&\disp\int_0^T\|\partial_{t}u_{\varepsilon}(\cdot,t)\|^{2}_{(W^{1,2}(\Omega))^*}dt\\
\leq&\disp{C_{29}\int_0^T\int_\Omega|\nabla u_{\varepsilon}|^{2}+C_{29}\int_0^T\int_\Omega |u_\varepsilon|^{\frac{10}{3}}+C_{29}\int_0^T\int_\Omega n_{\varepsilon}^{\frac{8(m-1)}{3}}+C_{29}T}\\
\leq&\disp{C_{30}(T+1)
~~\mbox{for all}~~ T>0.}\\
\end{array}
\label{gbhncvbmdcfxxxcvxxcvvbddfghnxdddffckkvbgtyyiiobddfffbvgcz2.5ghju48}
\end{equation}
and
some positive constants $C_{29}$ and $C_{30}.$
Finally,  in conjunction with \dref{gbhncvbmdcfxxdffxcvbnxdcyyuiivbbbvgcz2.5ghjddvccvvvbu48}--\dref{gbhncvbmdcfxxxcvxxcvvbddfghnxdddffckkvbgtyyiiobddfffbvgcz2.5ghju48}, we can get the results.
\end{proof}
In the following Lemma, we shall give some spatial estimates on $n_{\varepsilon}F_{\varepsilon}(n_{\varepsilon})S_\varepsilon(x, n_{\varepsilon}, c_{\varepsilon})\cdot\nabla c_{\varepsilon}$ and $u_\varepsilon\cdot\nabla c_\varepsilon$, which is crucial to derive the existence of weak solution to problem \dref{1.1}.
\begin{lemma}\label{4455lemma45630hhuujjuuyytt}
Assume that
 \begin{equation}
 \gamma_1:=\left\{\begin{array}{ll}
 {\frac{4(3m+2)}{3m+14}},~~\mbox{if}~~\frac{10}{9}< m\leq2,\\
 {\frac{8(m-1)}{4m-1}},~~\mbox{if}~~m>2\\
   \end{array}\right.
   \label{1.1dddfgccwwsqqaabddcfvgbhhhnjmdfgeddvbnmklllhyussddisda}
\end{equation}
   and
   \begin{equation}
 \gamma_2:=\left\{\begin{array}{ll}
 \frac{20}{11},~~\mbox{if}~~\frac{10}{9}< m\leq2,\\
  \frac{5}{4},~~\mbox{if}~~m>2.\\
   \end{array}\right.
  \label{1.1dddfgbddcfvgbhhhnjmdfgeddvbnmklllhyussddisda}
\end{equation}
Let $m>\frac{10}{9}$,
\dref{dd1.1fghyuisdakkkllljjjkk} and \dref{ccvvx1.731426677gg}
 hold.
 Then for any $T>0, $
  one can find $C > 0$ independent of $\varepsilon$ such that 
\begin{equation}
 \begin{array}{ll}
  \disp\int_0^T\int_{\Omega}|n_{\varepsilon}F_{\varepsilon}(n_{\varepsilon})S_\varepsilon(x, n_{\varepsilon}, c_{\varepsilon})\cdot\nabla c_{\varepsilon}|^{\gamma_1} \leq C(T+1)\\
   \end{array}\label{1.1dddfgbhnjmdfgeddvbnmklllhyussddisda}
\end{equation}
and
\begin{equation}
 \begin{array}{ll}
  \disp\int_0^T\int_{\Omega}|u_\varepsilon\cdot\nabla c_\varepsilon|^{\gamma_2}   \leq C(T+1).\\
   \end{array}\label{1.1dddfgbhnjmdfgeddvbnmklllhyussdddfgggdisda}
\end{equation}
\end{lemma}
\begin{proof}
Case $\frac{10}{9}<m\leq 2:$
Due to \dref{x1.73142vghf48gg}, \dref{gghhzjscz2.ddffg5297x9630111kkhhiioott4},
 \dref{gghhzjscz2.5297x9630111kkhhiioottvvbb}, \dref{bnmbncz2.htt678hyugghiihjj}
   and   the H\"{o}lder inequality, we derive that  there exist positive
constants $C_{1}$ and $C_2$ such that 
\begin{equation}
\begin{array}{rl}
&\disp\int_{0}^T\disp\int_{\Omega}|n_{\varepsilon}F_{\varepsilon}(n_{\varepsilon})S_\varepsilon(x, n_{\varepsilon}, c_{\varepsilon})\cdot\nabla c_{\varepsilon}|^{\frac{4(3m+2)}{3m+14}}\\
\leq&\disp{S_0(\|c_0\|_{L^\infty(\Omega)})^{\frac{4(3m+2)}{3m+14}}
\left(\int_0^T\int_{\Omega}|\nabla c_{\varepsilon}|^{4}\right)^{\frac{3m+2}{3m+14}}
\left(\int_0^T\int_{\Omega}n_{\varepsilon}^{\frac{3m+2}{3}}\right)^{\frac{12}{3m+14}}}\\
\leq&\disp{C_{1}(T+1)~~\mbox{for all}~~ T > 0}\\
\end{array}
\label{ddfff5555ddffbnmbncz2ddfvgffgtyybhh.htt678ghhjjjddfghhhyuiihjj}
\end{equation}
and
\begin{equation}
\begin{array}{rl}
\disp\int_{0}^T\disp\int_{\Omega}|u_\varepsilon\cdot\nabla c_\varepsilon|^{\frac{20}{11}}
\leq&\disp{\left(\int_0^T\int_{\Omega}|\nabla c_{\varepsilon}|^{4}\right)^{\frac{5}{11}}
\left(\int_0^T\int_{\Omega}|u_{\varepsilon}|^{\frac{10}{3}}\right)^{\frac{6}{11}}}\\
\leq&\disp{C_{2}(T+1)~~\mbox{for all}~~ T > 0.}\\
\end{array}
\label{ddfff5555ddffbnmbncz2ddfvgffgddfftyybhh.httff678ghhjjjddfghhhyuiihjj}
\end{equation}

Case $m> 2:$
In view of  \dref{x1.73142vghf48gg}, \dref{gghhzjscz2.ddffg5297x9630111kkhhiioott4},
 \dref{bnmbncz2.5ghhjuggyuuyuivvffggbnnihjj}
   and   the H\"{o}lder inequality, it follows that  there exist positive
constants $C_{3}$ and $C_4$ such that 
\begin{equation}
\begin{array}{rl}
&\disp\int_{0}^T\disp\int_{\Omega}|n_{\varepsilon}F_{\varepsilon}(n_{\varepsilon})S_\varepsilon(x, n_{\varepsilon}, c_{\varepsilon})\cdot\nabla c_{\varepsilon}|^{\frac{8(m-1)}{4m-1}}\\
\leq&\disp{S_0(\|c_0\|_{L^\infty(\Omega)})^{\frac{8(m-1)}{4m-1}}
\left(\int_0^T\int_{\Omega}|\nabla c_{\varepsilon}|^{2}\right)^{\frac{4(m-1)}{4m-1}}
\left(\int_0^T\int_{\Omega}n_{\varepsilon}^{\frac{8(m-1)}{3}}\right)^{\frac{3}{4m-1}}}\\
\leq&\disp{C_{3}(T+1)~~\mbox{for all}~~ T > 0}\\
\end{array}
\label{ddfff5555ddffbnmbnffrtcz2ddfvgffgtyybhh.htt678ghhjjjddfghhhyuiihjj}
\end{equation}
and
\begin{equation}
\begin{array}{rl}
\disp\int_{0}^T\disp\int_{\Omega}|u_\varepsilon\cdot\nabla c_\varepsilon|^{\frac{5}{4}}
\leq&\disp{\left(\int_0^T\int_{\Omega}|\nabla c_{\varepsilon}|^{2}\right)^{\frac{5}{8}}
\left(\int_0^T\int_{\Omega}|u_{\varepsilon}|^{\frac{10}{3}}\right)^{\frac{3}{8}}}\\
\leq&\disp{C_{4}(T+1)~~\mbox{for all}~~ T > 0.}\\
\end{array}
\label{ddfff5555ddffbnddmbnczgghh2ddfvgffgddfftyybhh.httff678ghhjjjddfghhhyuiihjj}
\end{equation}
Finally,  combining \dref{ddfff5555ddffbnmbncz2ddfvgffgtyybhh.htt678ghhjjjddfghhhyuiihjj}--\dref{ddfff5555ddffbnddmbnczgghh2ddfvgffgddfftyybhh.httff678ghhjjjddfghhhyuiihjj}, we can derive  \dref{1.1dddfgbhnjmdfgeddvbnmklllhyussddisda} and \dref{1.1dddfgbhnjmdfgeddvbnmklllhyussdddfgggdisda}.
This completes the proof of Lemma \ref{4455lemma45630hhuujjuuyytt}.

\end{proof}

\section{Passing to the limit. Proof of Theorem  \ref{theorem3}}
With the a-priori estimates obtained in Section 2 and Section 4, we shall give the proof of Theorem  \ref{theorem3}.
Before going to do it, 
let us first give the definition of weak solution.
In what follows, for vectors $v\in \mathbb{R}^3$ and $w\in \mathbb{R}^3$, we use  $v\otimes w$ denote the matrix $(a_{ij})_{i,j\in\{1,2,3\}}\in
\mathbb{R}^{3\times3}$ with $a_{ij}:=v_iw_j$ for $i, j\in\{1, 2, 3\}$.
\begin{definition} \label{df1} 
We call $(n, c, u)$ a global weak solution of \dref{1.1} if
\begin{equation}
 \left\{\begin{array}{ll}
   n\in L_{loc}^1(\bar{\Omega}\times[0,\infty)),\\
    c \in  L_{loc}^1([0,\infty);W^{1,1}(\Omega)),\\
u \in  L_{loc}^1([0,\infty); W^{1,1}_0(\Omega;\mathbb{R}^3)),\\
 \end{array}\right.\label{dffff1.1fghyuisdakkklll}
\end{equation}
such that  $n\geq 0$ and $c\geq 0$ a.e. in
$\Omega\times(0, \infty)$,
\begin{equation}\label{726291hh}
\begin{array}{rl}
&nc~\in L^1_{loc}(\bar{\Omega}\times [0, \infty)),~~~u\otimes u \in L^1_{loc}(\bar{\Omega}\times [0, \infty);\mathbb{R}^{3\times 3}),~~\mbox{and}\\
&D(n)\nabla n,~~ nS(x,n,c)\nabla c,~~~cu~~ \mbox{and}~~~ nu~~ \mbox{belong to}~~
L^1_{loc}(\bar{\Omega}\times [0, \infty);\mathbb{R}^{3}),\\
\end{array}
\end{equation}
that $\nabla\cdot u = 0$ a.e. in
$\Omega\times(0, \infty)$, and that
\begin{equation}
\begin{array}{rl}\label{eqx45xx12112ccgghh}
\disp{-\int_0^{T}\int_{\Omega}n\varphi_t-\int_{\Omega}n_0\varphi(\cdot,0)  }=&\disp{-
\int_0^T\int_{\Omega}D(n)\nabla n\cdot\nabla\varphi+\int_0^T\int_{\Omega}n(S(x,n,c)\cdot
\nabla c)\cdot\nabla\varphi}\\
&+\disp{\int_0^T\int_{\Omega}nu\cdot\nabla\varphi}\\
\end{array}
\end{equation}
for any $\varphi\in C_0^{\infty} (\bar{\Omega}\times[0, \infty))$
  as well as
  \begin{equation}
\begin{array}{rl}\label{eqx45xx12112ccgghhjj}
\disp{-\int_0^{T}\int_{\Omega}c\varphi_t-\int_{\Omega}c_0\varphi(\cdot,0)  }=&\disp{-
\int_0^T\int_{\Omega}\nabla c\cdot\nabla\varphi-\int_0^T\int_{\Omega}nc\cdot\varphi+
\int_0^T\int_{\Omega}cu\cdot\nabla\varphi}\\
\end{array}
\end{equation}
for any $\varphi\in C_0^{\infty} (\bar{\Omega}\times[0, \infty))$  and
\begin{equation}
\begin{array}{rl}\label{eqx45xx12112ccgghhjjgghh}
\disp{-\int_0^{T}\int_{\Omega}u\varphi_t-\int_{\Omega}u_0\varphi(\cdot,0)  }=&\disp{-
\int_0^T\int_{\Omega}\nabla u\cdot\nabla\varphi-
\int_0^T\int_{\Omega}n\nabla\phi\cdot\varphi}\\
\end{array}
\end{equation}
for any $\varphi\in C_0^{\infty} (\Omega\times[0, \infty);\mathbb{R}^3)$ fulfilling
$\nabla\varphi\equiv 0$.
\end{definition}

With the above compactness properties at hand, by means of a standard extraction procedure we can
conclude that \dref{1.1} is indeed globally solvable.
%
%
%
\begin{lemma}\label{lemma45630223}
Assume that
\dref{dd1.1fghyuisdakkkllljjjkk} and \dref{ccvvx1.731426677gg}
 hold, and suppose that $m$ and $S$ satisfy \dref{ghnjmk9161gyyhuug} and \dref{x1.73142vghf48rtgyhu}--\dref{x1.73142vghf48gg}, respectively.
If   $m> \frac{10}{9}$,
 then there exists $(\varepsilon_j)_{j\in \mathbb{N}}\subset (0, 1)$ such that $\varepsilon_j\searrow 0$ as $j\rightarrow\infty$, and such that as $\varepsilon: = \varepsilon_j\searrow 0$
we have
\begin{equation} n_\varepsilon\rightarrow n ~~\mbox{a.e.}~~ \mbox{in}~~ \Omega\times (0,\infty),
\label{z666jscz2.5297x9630222222}
\end{equation}
\begin{equation}
c_\varepsilon\rightarrow c ~~\mbox{in}~~ L^{2}(\bar{\Omega}\times[0,\infty))~~\mbox{and}~~\mbox{a.e.}~~\mbox{in}~~\Omega\times(0,\infty),
 \label{zjscz2.5297x963ddfgh06662222tt3}
\end{equation}
\begin{equation}
u_\varepsilon\rightarrow u~~\mbox{in}~~ L_{loc}^2(\bar{\Omega}\times[0,\infty))~~\mbox{and}~~\mbox{a.e.}~~\mbox{in}~~\Omega\times(0,\infty),
 \label{zjscz2.5297x96302222t666t4}
\end{equation}
\begin{equation}
 \nabla u_\varepsilon\rightharpoonup \nabla u ~~\mbox{ in}~~L^{2}(\bar{\Omega}\times[0,\infty)),
 \label{zjscz2.5297x96366602222tt4455}
\end{equation}
\begin{equation}
\nabla c_\varepsilon\rightharpoonup \nabla c~~\left\{\begin{array}{ll}
\mbox{in}~~ L_{loc}^{4}(\bar{\Omega}\times[0,\infty)),~~\mbox{if}~~\frac{10}{9}< m\leq2,\\
 \mbox{in}~~ L_{loc}^{2}(\bar{\Omega}\times[0,\infty)),~~\mbox{if}~~m>2,\\
   \end{array}\right.\label{1.1ddfgghhhge666ccdf2345ddvbnmklllhyuisda}
\end{equation}
 \begin{equation}
 n_\varepsilon\rightharpoonup n~~ \left\{\begin{array}{ll}
\mbox{in}~~ L_{loc}^{\frac{3m+2}{3}}(\bar{\Omega}\times[0,\infty)),~~\mbox{if}~~\frac{10}{9}< m\leq2,\\
 \mbox{in}~~ L_{loc}^{\frac{8(m-1)}{3}}(\bar{\Omega}\times[0,\infty)),~~\mbox{if}~~m>2,\\
   \end{array}\right.\label{1.1666ddfgeccdf2345ddvbnmklllhyuisda}
\end{equation}
\begin{equation}
D_{\varepsilon}(n_{\varepsilon})\nabla n_{\varepsilon}\rightharpoonup D(n)\nabla n~~\left\{\begin{array}{ll}
\mbox{in}~~ ~ L_{loc}^{\frac{3m+2}{3m+1}}(\bar{\Omega}\times[0,\infty)),~~\mbox{if}~~\frac{10}{9}< m\leq2,\\
\mbox{in}~~~L_{loc}^{\frac{8(m-1)}{4m-1}}(\bar{\Omega}\times[0,\infty)),~~\mbox{if}~~m>2,\\
   \end{array}\right.\label{1.1666ddccvvfggfggffgyhhhhhgeccdf2345ddvbnmklllhyuisda}
\end{equation}
\begin{equation}
c_\varepsilon\stackrel{*}{\rightharpoonup} c ~~\mbox{in}~~ L^{\infty}(\Omega\times(0,\infty))
 \label{zjscz2.5297x96302222666tt3}
\end{equation}
as well as
\begin{equation}
u_\varepsilon\rightharpoonup u ~~\mbox{in}~~ L^{\frac{10}{3}}(\bar{\Omega}\times[0,\infty))
 \label{zjscz2.5297x66696302222tt4}
\end{equation}
and
\begin{equation}
Y_\varepsilon u_\varepsilon\rightarrow u ~~\mbox{in}~~ L_{loc}^2([0,\infty); L^2(\Omega))
 \label{zjscz2.5297x96302266622tt44}
\end{equation}
 with some triple $(n, c, u)$ which is a global weak solution of \dref{1.1} in the sense of Definition \ref{df1}.
\end{lemma}
\begin{proof}

Firstly,  letting
$$
 \beta_1:=\left\{\begin{array}{ll}
 \frac{3m+2}{4},~~\mbox{if}~~\frac{10}{9}< m\leq2,\\
  2,~~\mbox{if}~~m>2,\\
   \end{array}\right.
   $$
   $$
 \gamma:=\left\{\begin{array}{ll}
 1,~~\mbox{if}~~\frac{10}{9}< m\leq2,\\
  m-1,~~\mbox{if}~~m>2,\\
   \end{array}\right.
   $$
$$
 W_1:=\left\{\begin{array}{ll}
 W^{2,q}(\Omega),~~\mbox{if}~~\frac{10}{9}< m<\frac{26}{21}~~\mbox{or}~~m>2,\\
W^{1,3m+2}(\Omega)~~\mbox{if}~~\frac{26}{21}\leq m\leq2,\\
   \end{array}\right.
   $$
$$
 \beta_2:=\left\{\begin{array}{ll}
 4,~~\mbox{if}~~\frac{10}{9}< m\leq2,\\
  2,~~\mbox{if}~~m>2\\
   \end{array}\right.
   $$
    as well as
   $$
 W_2:=\left\{\begin{array}{ll}
W^{1,\frac{3m+2}{3m-1}}(\Omega)~~\mbox{if}~~\frac{10}{9}< m\leq2,\\
  W^{1,2}(\Omega),~~\mbox{if}~~m>2\\
   \end{array}\right.
   $$
   and
   $$
 W_3:=\left\{\begin{array}{ll}
W^{1,\frac{3m+2}{3m-1}}(\Omega)~~\mbox{if}~~\frac{10}{9}< m<\frac{4}{3},\\
  W^{1,2}(\Omega),~~\mbox{if}~~m\geq\frac{4}{3},\\
   \end{array}\right.
   $$
   where $q$ is  given by Lemma \ref{lemma45630hhuujjuuyytt}.
%
%
%
Now,
in light of Lemma \ref{lemmakkllgg4563025xxhjklojjkkk}, Lemma \ref{lemmaghjssddgghhmk4563025xxhjklojjkkk} and Lemma \ref{lemma45630hhuujjuuyytt}, for some $C_1> 0$ which is independent of $\varepsilon$, we have
\begin{equation}
\begin{array}{rl}
\|n_{\varepsilon}^{\gamma}\|_{L^{\beta_1}_{loc}([0,\infty); W^{1,\beta_1}(\Omega))}\leq C_1(T+1)~~~\mbox{and}~~~\|\partial_{t}n_{\varepsilon}^{\gamma}\|_{L^{1}_{loc}([0,\infty); W_1^*)}\leq C_1(T+1)
\end{array}
\label{ggjjssdffzcddffcfccvvfggvvvvgccvvvgjscz2.5297x963ccvbb111kkuu}
\end{equation}
as well as
\begin{equation}
\begin{array}{rl}
\|c_{\varepsilon}\|_{L^{2}_{loc}([0,\infty); W^{1,2}(\Omega))}\leq C_1(T+1)~~~\mbox{and}~~~\|\partial_{t}c_{\varepsilon}\|_{L^{1}_{loc}([0,\infty); W_2^*)}\leq C_1(T+1)
\end{array}
\label{ggjjssdffzcddffcfccvvfggvvddffvvgccvvvgjscz2.5297x963ccvbb111kkuu}
\end{equation}
and
\begin{equation}
\begin{array}{rl}
\|u_{\varepsilon}\|_{L^{2}_{loc}([0,\infty); W^{1,2}(\Omega))}\leq C_1(T+1)~~~\mbox{and}~~~\|\partial_{t}u_{\varepsilon}\|_{L^{1}_{loc}([0,\infty); W_3^*)}\leq C_1(T+1).
\end{array}
\label{ggjjssdffzcddffcfccvvfggvvvvgccffghhvvvgjscz2.5297x963ccvbb111kkuu}
\end{equation}
Now, applying  the Aubin-Lions lemma (\cite{Simon})  to \dref{ggjjssdffzcddffcfccvvfggvvvvgccvvvgjscz2.5297x963ccvbb111kkuu}--\dref{ggjjssdffzcddffcfccvvfggvvvvgccffghhvvvgjscz2.5297x963ccvbb111kkuu}, we can derive that
\begin{equation}
\begin{array}{rl}
(n_{\varepsilon}^\gamma)_{\varepsilon\in(0,1)}~~~\mbox{is strongly precompact in}~~~L^{\beta_1}_{loc}(\bar{\Omega}\times[0,\infty))
\end{array}
\label{ggjjssdffzcddffcfcffffghcvvfggvvvvgccvvvgjscz2.5297x963ccvbb111kkuu}
\end{equation}
as well as
\begin{equation}
\begin{array}{rl}
(c_{\varepsilon})_{\varepsilon\in(0,1)}~~~\mbox{is strongly precompact in}~~~L^{2}_{loc}(\bar{\Omega}\times[0,\infty))
\end{array}
\label{ggjjssdffzcddddffcfcchhvvfggvvddffvvgccvvvgjscz2.5297x963ccvbb111kkuu}
\end{equation}
and
\begin{equation}
\begin{array}{rl}
(u_{\varepsilon})_{\varepsilon\in(0,1)}~~~\mbox{is strongly precompact in}~~~L^{2}_{loc}(\bar{\Omega}\times[0,\infty)).
\end{array}
\label{ggjjssdffzcccddffcfccvvfggvvggvvgccffghhvvvgjscz2.5297x963ccvbb111kkuu}
\end{equation}
Therefore, there exist a subsequence $\varepsilon=\varepsilon_j\subset (0,1)_{j\in \mathbb{N}}$
and the limit functions $n$ and $c$
such that \dref{zjscz2.5297x963ddfgh06662222tt3}--\dref{1.1ddfgghhhge666ccdf2345ddvbnmklllhyuisda} holds.
Moreover, for each fixed $T\in(0, \infty)$, \dref{zjscz2.5297x96302222t666t4} implies that there exists a   null set $N_T\in(0, T)$
such that we  can pick a subsequence which we still denote by $(\varepsilon_j)_{j\in N}$ fulfilling
\begin{equation}
u_\varepsilon(\cdot,t)\rightarrow u(\cdot,t) ~~\mbox{in}~~ L^{2}(\Omega)~~~\mbox{for all}~~~ t\in(0,T )\backslash N_T~~\mbox{as}~~ \varepsilon = ¦Å\varepsilon_j\searrow0.
 \label{zjscz2.5fddcfrrtyuuifgtt297xddssdrtt963022334466622tt4}
\end{equation}

Next, in view of \dref{ggjjssdffzcddffcfccvvfggvvvvgccvvvgjscz2.5297x963ccvbb111kkuu}, an Aubin--Lions lemma (see e.g.
\cite{Simon}) applies to yield strong precompactness of $(n_\varepsilon^{\gamma})_{\varepsilon\in(0,1)}$ in
$L^{\beta_1}(\Omega\times(0,T)),$ whence along a suitable subsequence we may derive that $n_\varepsilon^{\gamma}\rightarrow z^{\gamma}_1$
 and hence
 \begin{equation}
 n_{\varepsilon}\rightarrow z_1~~\mbox{ a.e. in}~~\Omega\times(0,\infty)~~\mbox{for some nonnegative measurable}~~z_1 : \Omega\times(0,\infty)\rightarrow \mathbb{R}.
 \label{z666jscz2.5297ffggxddrfff9630222222}
\end{equation}
The above estimate \dref{z666jscz2.5297ffggxddrfff9630222222} combined with energy inequality \dref{ggjjssdffzcddffcfcffffghcvvfggvvvvgccvvvgjscz2.5297x963ccvbb111kkuu},
\dref{z666jscz2.5297ffggxddrfff9630222222}, and the Egorov
theorem ensures
 $z_1 = n,$ therefore, we deduce that \dref{z666jscz2.5297x9630222222}.
Now, let
$$
 \beta_3:=\left\{\begin{array}{ll}
 \frac{3m+2}{3},~~\mbox{if}~~\frac{10}{9}< m\leq\frac{38}{33},\\
  \gamma_2,~~\mbox{if}~~m>\frac{38}{33},\\
   \end{array}\right.
   $$
   where $\gamma_2$ is given by \dref{1.1dddfgbddcfvgbhhhnjmdfgeddvbnmklllhyussddisda}.
On the other hand, observing that
$$
 \beta_3\leq\left\{\begin{array}{ll}
 \frac{3m+2}{3},~~\mbox{if}~~\frac{10}{9}< m\leq 2,\\
  \frac{8(m-1)}{3},~~\mbox{if}~~m>\frac{38}{33},\\
   \end{array}\right.
   $$
in light of
 \dref{hnjmssddaqwswddaassffssff3.10deerfgghhjuuloollgghhhyhh}, \dref{gghhzjscz2.5297x9630111kkhhiioottvvbb}, \dref{bnmbncz2.5ghhjuggyuuyuivvffggbnnihjj}, applying the Young inequality,
 we derive that  there exists a positive
constant $C_{1}$ such that
 \begin{equation}
 \begin{array}{ll}
  \disp\int_0^T\int_{\Omega}|n_\varepsilon c_\varepsilon|^{\beta_3}   \leq C_1(T+1).\\
   \end{array}\label{1.1dccvvvddfffggbhnjmdfgeddvbnmklllhyussdddfgggdisda}
\end{equation}

Next, let $g_\varepsilon(x, t) := -n_{\varepsilon}c_\varepsilon-u_{\varepsilon}\cdot\nabla c_{\varepsilon}.$
Therefore,  recalling $1<\beta_3\leq\gamma_2$, by some basic calculation, we  can get
\begin{equation}
 \begin{array}{ll}
  \disp\int_0^T\int_{\Omega}|g_\varepsilon|^{\beta_3}   \leq C_2(T+1).\\
   \end{array}\label{1.1dccvvvddfffggbhnjmdfgeddvbnmklddfggllhyussdddfgggdisda}
\end{equation}
for a positive constant $C_2.$
From this,
 we may invoke the standard parabolic regularity theory  to infer that
$(c_{\varepsilon})_{\varepsilon\in(0,1)}$ is bounded in
$L^{\beta_3} ((0, T); W^{2,\beta_3}(\Omega))$.
Thus,  by \dref{1.1dddfgbhnjmdfgeddvbnmklllhyussddisda} and the Aubin--Lions lemma we derive that  the relative compactness of $(c_{\varepsilon})_{\varepsilon\in(0,1)}$ in
$L^{\beta_3} ((0, T); W^{1,\beta_3}(\Omega))$. We can pick an appropriate subsequence which is
still written as $(\varepsilon_j )_{j\in \mathbb{N}}$ such that $\nabla c_{\varepsilon_j} \rightarrow z_2$
 in $L^{\beta_3} (\Omega\times(0, T))$ for all $T\in(0, \infty)$ and some
$z_2\in L^{\beta_3} (\Omega\times(0, T))$ as $j\rightarrow\infty$, hence $\nabla c_{\varepsilon_j} \rightarrow z_2$ a.e. in $\Omega\times(0, \infty)$
 as $j \rightarrow\infty$.
In view  of \dref{1.1ddfgghhhge666ccdf2345ddvbnmklllhyuisda} and  the Egorov theorem we conclude  that
$z_2=\nabla c,$ and whence
\begin{equation}
\nabla c_\varepsilon\rightarrow \nabla c~~\begin{array}{ll}
  ~\mbox{a.e.}~~\mbox{in}~~\Omega\times(0,\infty)~~~\mbox{as}~~\varepsilon =\varepsilon_j\searrow0.
   \end{array}\label{1.1ddhhyujiiifgghhhge666ccdf2345ddvbnmklllhyuisda}
\end{equation}
Now, we will conclude that the triplet $(n, c, u)$ is the desired solution in the sense of Definition \ref{df1}. Indeed, we first notice that from the nonnegativity of $n_\varepsilon$ and $c_\varepsilon$, the estimate of \dref{zjscz2.5297x96366602222tt4455} and $\nabla\cdot u_{\varepsilon} = 0$, it is easy to see  that $n \geq 0$ and $c\geq 0$
and
$\nabla\cdot u = 0$ a.e. in $\Omega\times (0, \infty)$.
On the other hand, in view of \dref{1.1dddfgbhnjmdfgeddvbnmklllhyussddisda}, we can infer from
\dref{gghhzjscz2.5297x9630111kkhhiioottvvbb}
and
\dref{bnmbncz2.5ghhjuggyuuyuivvffggbnnihjj} that
\begin{equation}
n_{\varepsilon}F_{\varepsilon}(n_{\varepsilon})S_\varepsilon(x, n_{\varepsilon}, c_{\varepsilon})\nabla c_{\varepsilon}\rightharpoonup z_3~~\begin{array}{ll}
  ~~~\mbox{in}~~ L^{\gamma_1}(\Omega\times(0,T))~~\mbox{for each}~~ T \in(0, \infty)
   \end{array}
\label{1.1666ddffvgbhhfgeccdf2345ddvbnmklllhyuisda}
\end{equation}
   with $\gamma_1$ is given by \dref{1.1dddfgccwwsqqaabddcfvgbhhhnjmdfgeddvbnmklllhyussddisda}.
On the other hand, in view of  \dref{x1.73142vghf48rtgyhu}, \dref{1.ffggvbbnxxccvvn1}, \dref{z666jscz2.5297x9630222222}, \dref{zjscz2.5297x963ddfgh06662222tt3} and \dref{1.1ddhhyujiiifgghhhge666ccdf2345ddvbnmklllhyuisda} imply that
\begin{equation}
n_{\varepsilon}F_{\varepsilon}(n_{\varepsilon})S_\varepsilon(x, n_{\varepsilon}, c_{\varepsilon})\nabla c_{\varepsilon}\rightarrow
 nS(x, n, c)\nabla c
 \begin{array}{ll}
  ~~\mbox{a.e.}~~\mbox{in}~~\Omega\times(0,\infty)~~~\mbox{as}~~\varepsilon =\varepsilon_j\searrow0.
   \end{array}
   \label{1.1666ccccftyyuufvggddccvgffvgbhhfgeccdf2345ddvbnmklllhyuisda}
\end{equation}
Again by the Egorov theorem, we gain $z_3=nS(x, n, c)\nabla c,$ and hence \dref{1.1666ddffvgbhhfgeccdf2345ddvbnmklllhyuisda} can be rewritten as
\begin{equation}
n_{\varepsilon}F_{\varepsilon}(n_{\varepsilon})S_\varepsilon(x, n_{\varepsilon}, c_{\varepsilon})\nabla c_{\varepsilon}\rightharpoonup
 nS(x, n, c)\nabla c
 \begin{array}{ll}
  ~\mbox{in}~ L^{\gamma_1}(\Omega\times(0,T))~\mbox{for each}~~ T \in(0, \infty)
   \end{array}
   \label{1.1666ccfvggddccvgffvgbhhfgeccdf2345ddvbnmklllhyuisda}
\end{equation}
as $\varepsilon =\varepsilon_j\searrow0.$
Next, employing almost exactly the same arguments as in the proof of \dref{1.1666ddffvgbhhfgeccdf2345ddvbnmklllhyuisda}--\dref{1.1666ccfvggddccvgffvgbhhfgeccdf2345ddvbnmklllhyuisda} (the minor necessary changes are left as an easy exercise to the reader), and taking advantage of \dref{hnjmssddaqwswddaassffssff3.10deerfgghhjuuloollgghhhyhh}, \dref{gghhzjscz2.ddffg5297x9630111kkhhiioott4}--\dref{bnmbncz2.5ghhjuggyuuyuivvffggbnnihjj} and \dref{z666jscz2.5297x9630222222}--\dref{zjscz2.5297x96302222t666t4}, we conclude that
\dref{1.1ddfgghhhge666ccdf2345ddvbnmklllhyuisda}--\dref{zjscz2.5297x66696302222tt4} is true  as well as
\begin{equation}
c_\varepsilon u_\varepsilon\rightarrow cu ~~\mbox{ in}~~ L^{1}_{loc}(\bar{\Omega}\times(0,\infty))~~~\mbox{as}~~\varepsilon=\varepsilon_j\searrow0,
 \label{zxxcvvfgggjscddfffcvvfggz2.5297x96302222tt4}
\end{equation}
\begin{equation}
n_\varepsilon c_\varepsilon \rightharpoonup nc ~~\mbox{ in}~~ L^{\gamma_4}_{loc}(\bar{\Omega}\times(0,\infty))~~~\mbox{as}~~\varepsilon=\varepsilon_j\searrow0
 \label{zxxcvvfgggjscddfffcvvfggzddfgg2.5297x96302222tt4}
\end{equation}
and
\begin{equation}
n_\varepsilon u_\varepsilon\rightharpoonup nu ~~\mbox{ in}~~ L^{\gamma_5}(\Omega\times(0,T))~~~\mbox{as}~~\varepsilon=\varepsilon_j\searrow0
 \label{zxxcvvfgggjscddfffcvvfggz2.5297ffghhx96302222tt4}
\end{equation}
for each $T \in(0, \infty),$
where
\begin{equation}
 \gamma_4:=\left\{\begin{array}{ll}
 {\frac{1}{\frac{1}{4}+\frac{3}{3m+2}}},~~\mbox{if}~~\frac{10}{9}< m\leq2,\\
  {\frac{1}{\frac{1}{2}+\frac{3}{8(m-1)}}},~~\mbox{if}~~m>2\\
   \end{array}\right.
   \label{1.1dddfgccwwsddrttqqaabddcfvgbhhhffggnjmdfgeddvbnmklllhyussddisda}
\end{equation}
and
\begin{equation}
 \gamma_5:=\left\{\begin{array}{ll}
 {\frac{1}{\frac{3}{10}+\frac{3}{3m+2}}},~~\mbox{if}~~\frac{10}{9}< m\leq2,\\
  {\frac{1}{\frac{3}{10}+\frac{3}{8(m-1)}}},~~\mbox{if}~~m>2.\\
   \end{array}\right.
   \label{1.1dddfgccwwsddrttqqaabdddffggdcfvgbhhhffggnjmdfgeddvbnmklllhyussddisda}
\end{equation}
Now, by
\dref{zjscz2.5297x963ddfgh06662222tt3}--\dref{zjscz2.5297x96366602222tt4455}, \dref{1.1666ddfgeccdf2345ddvbnmklllhyuisda}, we conclude that  \dref{dffff1.1fghyuisdakkklll}.
Now, employing  \dref{zjscz2.5297x96302222t666t4} and
using the fact that
 $\|Y_{\varepsilon}\varphi\|_{L^2(\Omega)} \leq \|\varphi\|_{L^2(\Omega)}(\varphi\in L^2_{\sigma}(\Omega))$
and
$Y_{\varepsilon}\varphi \rightarrow \varphi$ in $L^2(\Omega)$ as $\varepsilon\searrow0$, we can  obtain
\begin{equation}
\begin{array}{rl}
\left\|Y_{\varepsilon}u_{\varepsilon}(\cdot,t)-u(\cdot,t)\right\|_{L^2(\Omega)}  \leq&\disp{\left\|Y_{\varepsilon}[u_{\varepsilon}(\cdot,t)-u(\cdot,t)]\right\|_{L^2(\Omega)}+
\left\|Y_{\varepsilon}u(\cdot,t)-u(\cdot,t)\right\|_{L^2(\Omega)}}\\
\leq&\disp{\left\|u_{\varepsilon}(\cdot,t)-u(\cdot,t)\right\|_{L^2(\Omega)}+
\left\|Y_{\varepsilon}u(\cdot,t)-u(\cdot,t)\right\|_{L^2(\Omega)}}\\
\rightarrow&\disp{0~~\mbox{as}~~\varepsilon:=\varepsilon_j\searrow0.}\\
\end{array}
\label{ggjjssdffzccvvvvggjscz2.5297x963ccvbb111kkuu}
\end{equation}
On the other hand,  observing that
\begin{equation}
\begin{array}{rl}
\left\|Y_{\varepsilon}u_{\varepsilon}(\cdot,t)-u(\cdot,t)\right\|_{L^2(\Omega)}^2  \leq&\disp{\left(\|Y_{\varepsilon}u_{\varepsilon}(\cdot,t)|\|_{L^2(\Omega)}+\|u(\cdot,t)|\|_{L^2(\Omega)}\right)^2}\\
\leq&\disp{\left(\|u_{\varepsilon}(\cdot,t)|\|_{L^2(\Omega)}+\|u(\cdot,t)|\|_{L^2(\Omega)}\right)^2}\\
\leq&\disp{C_2~~\mbox{for all}~~t\in(0,\infty)~~\mbox{and}~~\varepsilon\in(0,1) }\\
\end{array}
\label{ggjjssdffzccffggvvvvggjscz2.5297x963ccvbb111kkuu}
\end{equation}
with some $C_2 > 0$.
Now, thanks to \dref{zjscz2.5297x96302222t666t4}, \dref{ggjjssdffzccvvvvggjscz2.5297x963ccvbb111kkuu} and \dref{ggjjssdffzccffggvvvvggjscz2.5297x963ccvbb111kkuu} and the dominated convergence theorem, we conclude that
\begin{equation}
\begin{array}{rl}
\disp\int_{0}^T\|Y_{\varepsilon}u_{\varepsilon}(\cdot,t)-u(\cdot,t)\|_{L^2(\Omega)}^2dt\rightarrow0 ~~\mbox{as}~~\varepsilon:=\varepsilon_j\searrow0 ~~~\mbox{for all}~~T>0.
\end{array}
\label{ggjjssdffzccffggvvvvgccvvvgjscz2.5297x963ccvbb111kkuu}
\end{equation}
Thus, \dref{zjscz2.5297x96302266622tt44} holds.
Now, in conjunction with
  \dref{zjscz2.5297x96302266622tt44}  and \dref{zjscz2.5297x96302222t666t4}, we can obtain
\begin{equation}
\begin{array}{rl}
Y_{\varepsilon}u_{\varepsilon}\otimes u_{\varepsilon}\rightarrow u \otimes u ~~\mbox{in}~~L^1_{loc}(\bar{\Omega}\times[0,\infty))~~\mbox{as}~~\varepsilon:=\varepsilon_j\searrow0.
\end{array}
\label{ggjjssdffzccfccvvfgghjjjvvvvgccvvvgjscz2.5297x963ccvbb111kkuu}
\end{equation}
Therefore, \dref{1.1666ccfvggddccvgffvgbhhfgeccdf2345ddvbnmklllhyuisda}--\dref{zxxcvvfgggjscddfffcvvfggz2.5297ffghhx96302222tt4} and \dref{ggjjssdffzccfccvvfgghjjjvvvvgccvvvgjscz2.5297x963ccvbb111kkuu} imply the integrability of
$nS(x, n, c)\nabla c, nc,nu$ and $cu,u\otimes u$ in \dref{726291hh}.
Based on \dref{zjscz2.5297x96366602222tt4455}--\dref{zjscz2.5297x96302266622tt44},
\dref{1.1666ccfvggddccvgffvgbhhfgeccdf2345ddvbnmklllhyuisda}--\dref{zxxcvvfgggjscddfffcvvfggz2.5297ffghhx96302222tt4} and \dref{ggjjssdffzccfccvvfgghjjjvvvvgccvvvgjscz2.5297x963ccvbb111kkuu}, the integral
identities \dref{eqx45xx12112ccgghh}--\dref{eqx45xx12112ccgghhjjgghh}
can be achieved by standard arguments from the corresponding
weak formulations in the regularized system \dref{1.1fghyuisda} upon taking $\varepsilon= \varepsilon_j\searrow 0.$
The proof of Lemma \ref{lemma45630223} is completed.
\end{proof}

{\bf The proof of Theorem \ref{theorem3}}~
The statement is evidently implied by Lemma \ref{lemma45630223}.

{\bf Acknowledgement}:
This work is partially supported by the
Natural Science Foundation of Shandong Province of China (No. ZR2016AQ17 and No. ZR2015PA004), the National Natural
Science Foundation of China (No. 11601215),
 and the Doctor Start-up Funding of Ludong University (No. LA2016006).

\end{document}